\documentclass[12pt]{amsart}

\usepackage{amsmath,amssymb}
\usepackage{graphicx}
\usepackage{import}
\usepackage{amscd}
\usepackage{wrapfig}
\usepackage{fullpage}
\usepackage{epsfig}
\numberwithin{equation}{section}

\newtheorem{theorem}{Theorem}[section]

\newtheorem{lemma}[theorem]{Lemma}
\newtheorem{corollary}[theorem]{Corollary}
\newtheorem{proposition}[theorem]{Proposition}
\newtheorem{claim}[theorem]{Claim}

\theoremstyle{definition}
\newtheorem{definition}[theorem]{Definition}

\theoremstyle{remark}
\newtheorem{remark}[theorem]{Remark}

\newcommand{\D}{{\mathbb D}}

\newcommand{\N}{{\mathbb N}}
\newcommand{\R}{{\mathbb R}}
\newcommand{\Z}{{\mathbb Z}}

\renewcommand{\H}{{\mathbb H}}

\newcommand{\x}{{\mathbf{x}}}
\newcommand{\y}{{\mathbf{y}}}

\newcommand{\bb}{\mathbf b}

\newcommand{\J}{{\mathcal J}}
\newcommand{\p}{\omega}

\newcommand{\cross}{\times}

\newcommand{\mcg}{\mathcal{MC\kern0.04emG}}
\newcommand{\ml}{\mathcal{M \kern0.07emL}}
\newcommand{\mf}{\mathcal{M \kern0.07emF}}
\newcommand{\pml}{\mathcal{P \kern0.07em M \kern0.07em L}}
\newcommand{\pmf}{\mathcal{P \kern0.07em M \kern0.07em F}}
\newcommand{\ue}{\mathcal{U \kern0.07em E}}

\newcommand{\B}{\mathcal{B}}

\newcommand{\cG}{\mathcal{G}}
\newcommand{\cC}{\mathcal{C}}

\newcommand{\teichmuller}{Teichm{\"u}ller{ }}

\def\fmin{\mathcal{F}_{min}}

\makeatletter
 \let\c@theorem=\c@subsection
 \let\c@conjecture=\c@subsection
 \let\c@lemma=\c@subsection
 \let\c@proposition=\c@subsection
 \let\c@claim=\c@subsection
 \let\c@question=\c@subsection
 \let\c@criterion=\c@subsection
 \let\c@vfconj=\c@subsection
 \let\c@definition=\c@subsection
 \let\c@notation=\c@subsection
 \let\c@remark=\c@subsection
 \let\c@example=\c@subsection
 \let\c@equation=\c@subsection
 \let\c@figure=\c@subsection
 \let\c@wrapfigure=\c@subsection

\makeatother

\begin{document}

\title[harmonic measure]{Harmonic measures for distributions with finite support on the mapping class group are singular}
\author[Gadre]{Vaibhav S Gadre}
\address{Mathematics 253-37, California Institute of Technology, Pasadena, CA 91125, USA}
\email{vaibhav@caltech.edu}


\begin{abstract}
Kaimanovich and Masur showed that a random walk on the mapping class group for an initial distribution with finite first moment and whose support generates a non-elementary subgroup, converges almost surely to a point in the space $\pmf$ of projective measured foliations on the surface. This defines a harmonic measure on $\pmf$. Here, we show that when the initial distribution has finite support, the corresponding harmonic measure is singular with respect to the natural Lebesgue measure on $\pmf$.
\end{abstract}

\maketitle

\section{Introduction} \label{Intro}
Let $\Sigma$ be an orientable surface of finite type. The mapping class group $G$ of $\Sigma$ is the group of orientation preserving diffeomorphisms of $\Sigma$ modulo those isotopic to identity. The \teichmuller space $T(\Sigma)$ of $\Sigma$ is the space of marked conformal structures on $\Sigma$ modulo biholomorphisms isotopic to identity. The mapping class group $G$ acts on $T(\Sigma)$ by changing the marking. The \teichmuller space is homeomorphic to an open ball in $\R^{6g-6+3m}$, where $g$ is the genus and $m$ is the number of punctures of $\Sigma$. Thurston showed that $T(\Sigma)$ can be compactified by the space $\pmf$ of projective measured foliations on $\Sigma$. The action of $G$ extends continuously to $\pmf$.

In \cite{Kai-Mas1}, Kaimanovich and Masur considered random walks on $G$ with some initial distribution $\mu$. It is possible to project the random walk into $T(\Sigma)$ by choosing a base-point and then using the action of $G$. They showed that if the subgroup of $G$ generated by the support of $\mu$ is non-elementary, then almost every sample path converges to some uniquely ergodic foliation in $\pmf$. This means that there is a well defined hitting measure $\nu$ on $\pmf$ coming from the random walk. Moreover, the measure $\nu$ is a harmonic measure in the sense that if $A$ is a measurable set of $\pmf$ then
\[
\nu(A) = \sum_{g \in G} \mu(g) \nu(g^{-1}A)
\]

Complete train tracks on $\Sigma$ define an atlas of charts on the space $\mf$ of measured foliations on $\Sigma$. The set of integer weights on a train track correspond to mutli-curves on $\Sigma$ carried by it. So the transition functions between charts have to preserve the integer weights. As a result, even though this process does not necessarily  give, after projectivization, a global {\em Lebesgue} measure on $\pmf$, there is still a well-defined Lebesgue measure class.

The main theorem here is

\begin{theorem}\label{singular}
If $\mu$ is a finitely supported initial probability distribution on $G$ such that the subgroup of $G$ generated by the support of $\mu$ is non-elementary, then the induced harmonic measure $\nu$ on $\pmf$ is singular with respect to the Lebesgue measure class.
\end{theorem}

The proof of the theorem works by actually constructing a singular set for $\nu$ i.e. a measurable set in $\pmf$ that has full Lebesgue measure and zero harmonic measure.

Conversely, one can ask if there is a Lebesgue measure $\ell$ on $\pmf$ such that $\ell$ is a harmonic measure for some initial distribution? If it exists, the distribution has infinite support as a corollary to our theorem. In fact, recently Eskin, Mirzakhani and Rafi have announced that such a distribution exists. In particular, they show that a certain Lebesgue measure coming from asymptotic volumes of extremal length balls arises as a harmonic measure for a random walk on $G$.

\subsection{General groups:} The study of boundary phenomena for random processes on groups was initiated by Furstenberg in the 60's when he showed that semi-simple non-compact Lie groups have a natural boundary for Brownian motion on them. Moreover, he showed that every lattice in the Lie group carries a random walk for which the measure theoretic boundary is the geometric boundary for the Brownian motion. This let to the first rigidity results. When a group is the fundamental group of a manifold with with a geometric boundary (for instance, non-positively curved manifolds), in addition to the harmonic measure for Brownian motion, there are two other naturally defined measures on the group boundary: the visual or Lebesgue measure, and the Patterson-Sullivan measure. In case of compact negatively curved manifolds, it is known that either the manifold is locally symmetric in which case the three measures coincide, or the measures are all mutually singular. For general groups, there are some known results.

In \cite{Gui-LeJ}, Guivarc'h and LeJan showed that the harmonic measure on $S^1$ given by a finitely supported random walk on the fundamental group of a surface with punctures but finite volume, is singular with respect to the Lebesgue measure on $S^1$. This result has been generalized to certain types of finitely generated groups of circle diffeomorphisms by Deroin, Kleptsyn and Navas \cite{Der-Kle-Nav}. In \cite{Lyo}, Lyons showed that there are examples of finitely supported random walks on universal covers of finite graphs such that associated harmonic measure coincides with the Patterson-Sullivan measure. In \cite{Kai-LeP}, Kaimanovich and LePrince show that any Zariski dense countable subgroup of $SL(n, \R)$ carries a non-degenerate finitely supported random walk such that the induced harmonic measure on the flag space is singular.

\subsection{Outline of the paper:} In Section~\ref{Teich}, we begin with some preliminaries from \teichmuller theory, and define the Lebesgue measure class on $\pmf$. In Section~\ref{Random}, we give some background on random walks on groups, and state Theorem~\ref{Poisson} due to Kaimanovich and Masur for mapping class groups, and the main theorem of this paper. In Section~\ref{Borel-Cantelli}, we state and prove the key measure theory result that we exploit in the construction of the singular set. In Section~\ref{Torus}, we explain the construction of the singular set in the special case of the torus, and indicate how the construction should generalize. In the process, we introduce the techniques of classical interval exchanges. In Section~\ref{Niem}, we consider non-classical interval exchanges, which provide charts on $\pmf$. We state the key technical theorem, Theorem~\ref{Uniform-distortion}. In Section~\ref{Combinatorics}, we construct non-classical interval exchanges with the particular combinatorics needed to carry out the construction. We also show the necessary estimates for the Lebesgue measure. In Section~\ref{Curvecomplex}, we provide some background on the curve complex, and state Klarreich's Theorem~\ref{Klarreich}. We also explain the freedom it allows us to choose the space to project the mapping class group random walk. In Section~\ref{Projections}, we introduce the marking complex and relative space, and provide some background on sub-surface projections defined on these spaces. In Section~\ref{Facts}, we outline the key facts due to Maher about half-spaces in the relative space. In Section~\ref{Maher}, we state and prove the main decay result, Theorem~\ref{exp-decay}, due to Maher. This allows us to estimate the decay of harmonic measure with respect to nesting under a sub-surface projection. In Section~\ref{Pushin}, we outline a technical trick that provides the setup necessary to apply Theorem~\ref{exp-decay}. Finally, in Section~\ref{Singular}, we put together all the ingredients to construct a singular set.

\subsection{Acknowledgements:} The research was supported by NSF graduate fellowship under Nathan Dunfield by grant \# 0405491 and \#0707136. This work was completed while the author was at University of Illinois, Urbana-Champaign. The author would like to thank his advisor, Nathan Dunfield, for extensive discussions and constant support. The author thanks Joseph Maher for numerous discussions, and giving permission to include Theorem~\ref{exp-decay} here. The author also thanks Saul Schleimer, Chris Connell and Chris Leininger for helpful conversations during the course of this work.

\section{Preliminaries from \teichmuller theory} \label{Teich}

Let $\Sigma_{g,m}$ be an orientable surface with genus $g$ and $m$ punctures. For brevity, we will drop the subscripts and call it just $\Sigma$. The group of orientation preserving diffeomorphisms of $\Sigma$ that preserve the set of punctures modulo those isotopic to identity, is called the mapping class group of $\Sigma$. Throughout, we shall denote the mapping class group by $G$.

\subsection{\teichmuller space:}
The space of marked conformal structures on $\Sigma$ modulo biholomorphisms isotopic to identity is called the \teichmuller space of $\Sigma$. By the uniformization theorem, modulo isometries isotopic to identity, there is a unique marked hyperbolic metric in each marked conformal class. This means that the \teichmuller space can also be thought of as the space of marked hyperbolic metrics on $\Sigma$ modulo isometries isotopic to identity. We denote the \teichmuller space by $T(\Sigma)$. The mapping class group $G$ acts on $T(\Sigma)$ by changing the marking.

It is a classical theorem that topologically, the \teichmuller space $T(\Sigma)$ is homeomorphic to an open ball in $\R^{6g-6+3m}$, where $g$ is the genus and $m$ is the number of punctures of $\Sigma$. Thurston showed that $T(\Sigma)$ can be compactified by the space $\pmf$ of projective measured foliations on $\Sigma$, such that the action of the mapping class group $G$ on $T(\Sigma)$ extends to a continuous action on the boundary, $\pmf$ of $T(\Sigma)$. The space $\pmf$ is homeomorphic to a sphere of dimension $6g-7+3m$. It is also the same as the space of projective measured laminations $\pml$ on the surface by a homeomorphism that is mapping class group equivariant.

\subsection{Lebesgue measure on $\pmf$:}
A train track on the surface $\Sigma$ is an embedded 1-dimensional CW complex in which there is a common line of tangency to all the 1-dimensional branches that join at a 0-dimensional switch. This splits the set of branches incident on a switch into two disjoint subsets, which can be arbitrarily assigned as outgoing edges and incoming edges. Train tracks whose complementary regions are either triangles or once punctured monogons are said to be {\em complete}. Here, we shall restrict only to complete train tracks.

One can assign non-negative weights to the branches of a train track so that the sum of the weights of the outgoing branches at a switch is equal to the sum of the weights of the incoming branches. Each such choice of weights compatible with the switch conditions, defines a measured foliation on $\Sigma$. A measured foliation defined this way is said to be {\em carried} by the train track. The set of measured foliations carried by a complete train track has the structure of a cone (i.e. a homogeneous set) in $\R^{6g-6+3m}_{\geqslant  0}$, over a convex polytope of dimension $6g-7+3m$ with finitely many extremal vertices. Moreover, the directions determined by the extremal vertices are rational; hence we can find a minimal integer point on each of them. These integer points represent simple closed curves on $\Sigma$, and are called the {\em vertex cycles} of the train track. In general, integer points in the cone represent multi-curves carried by the train track.

Complete train tracks on $\Sigma$ define an atlas of charts on the space $\mf$ of measured foliations on $\Sigma$. If we choose to projectivize by a normalization, such as the sum of weights is 1, the transition functions do not preserve this normalization. So we do not get a global measure on $\pmf$ by this process. However, the transition functions have to preserve the multi-curves i.e the integer points. Hence, there is a well defined measure class on $\pmf$. We call the measure class on $\pmf$ coming from these charts the {\em Lebesgue} measure class.

\section{Random walks}\label{Random}

Let $G$ be a group and $\mu$ a probability distribution on $G$. A random walk on $G$ is a Markov chain with transition probabilities $p(g,h)= \mu(g^{-1}h)$. It is assumed that one starts at the identity element in $G$ at time zero. Denote the group element generated in $n$ steps of the random walk by $\omega_n$ i.e. the group element $\omega_n$ is the product  $g_1g_2 \cdots g_n$ where each group element $g_i$ is sampled by $\mu$. The distribution of $\omega_n$ is given by the $n$-fold convolution $\mu^{(n)}$ of the starting distribution $\mu$. The {\em path space} for the random walk is the probability space $(G^{\Z_+}, \mathbb{P})$, where $G^{\Z_+}$ is the set of all one-sided infinite sequences of elements of $G$. The probability measure $\mathbb{P}$ is determined by the convolutions $\mu^{(n)}$ using the Kolmogorov extension theorem. The group $G$ acts on the path space on the left, as opposed to the increments $g_n$ from $\omega_{n-1}$ to $\omega_n$, which get multiplied on the right. For a general background about random walks on infinite groups, see \cite{Woe}.

From now on, let $G$ be the mapping class group. A subgroup of $G$ is {\em non-elementary} if it contains a pair of psuedo-Anosov elements with distinct stable and unstable measured foliations. Kaimanovich and Masur proved the following theorem in \cite{Kai-Mas1}

\begin{theorem} \label{Poisson}
If $\mu$ is a probability measure on the mapping class group $G$ such that the group generated by its support is non-elementary, then there exists a unique $\mu$-stationary probability measure $\nu$ on $\pmf$, which is purely non-atomic and concentrated on the subset $\ue \subset \pmf$ of uniquely ergodic foliations. For any $X \in T(\Sigma)$, and almost every sample path $\p = \{ \omega_n\}$ determined by $(G,\mu)$, the sequence $\omega_n X$ converges to a limit $F(\p)$ in $\ue$, and the distribution of the limit points $F(\p)$ is given by $\nu$.
\end{theorem}
The measure $\nu$ on $\pmf$ is $\mu$-stationary in the sense that if $A$ is a measurable set in $\pmf$ then the measure $\nu$ satisfies
\begin{equation}
\nu(A)= \sum_{g \in G} \mu(g) \nu(g^{-1} A)
\end{equation}
The measure $\nu$ is called a {\em harmonic} measure because of the above property.

\begin{remark}
Kaimanovich and Masur state the result only for closed surfaces, but the proof works for surfaces with punctures also, as pointed out in \cite{Far-Mas}. Kaimanovich and Masur also show that if, in addition to the hypothesis of Theorem~\ref{Poisson}, the initial distribution $\mu$ has finite entropy and finite first logarithmic moment with respect to the \teichmuller metric, then the measure space $(\pmf, \nu)$ is the Poisson boundary of $(G, \mu)$.
\end{remark}

The main theorem we prove is

\begin{theorem}\label{singular}
If $\mu$ is a finitely supported probability distribution on the mapping class group $G$ such that the subgroup of $G$ generated by the support is non-elementary, then the induced harmonic measure $\nu$ on $\pmf$ is singular with respect to the Lebesgue measure class.
\end{theorem}

To construct a singular set for $\nu$ i.e. a measurable set of $\pmf$ that has full Lebesgue measure but zero harmonic measure, the crucial point is to understand the action of the reducible elements, say Dehn twists, from the measure theoretic point of view. We construct complete train tracks on $\Sigma$ with the property that a positive Dehn twist in one of its vertex cycles can be realized as a splitting sequence of the train track. After normalizing so that the measures of the original chart are 1, we show that the Lebesgue measure of the charts obtained after applying the Dehn twist splitting sequence $n$ successive times, are $\approx 1/n^k$, for some positive integer $k$. On the other hand, by a theorem of Maher, the harmonic measures of the same charts, are $\leqslant  \exp (-n)$. It is this discrepancy that we exploit to give a construction of the singular set. The key measure theoretic tool is the slightly generalized Borel-Cantelli lemma, namely Proposition~\ref{BC2}, that we state in the next section.

\section{The Borel-Cantelli Setup} \label{Borel-Cantelli}

In this section, we state and prove the key theorem in measure theory that we use to show that our construction in Section~\ref{Singular} gives us a singular set.

We state a version \cite{Str} of the Borel-Cantelli lemma that generalizes the classical Borel-Cantelli lemma to the case when the sequence of events are {\em pairwise almost independent} instead of independent.

\begin{lemma}[Borel-Cantelli] \label{BC1}
Let $(\Omega, \B, \mu)$ be a probability space and $\{ X_n\}_{1}^{\infty}$ be a sequence of $\B$-measurable sets such that there exists a positive integer $d$ and a constant $c>1$ for which
\begin{equation}\label{pairwise-ind}
\mu(X_m \cap X_n) < c \mu(X_m) \mu(X_n), \hspace{5mm} m \in \N \hspace{2mm} \text{and} \hspace{2mm} n \geqslant m+d
\end{equation}
Then
\[
\sum_{n=1}^{\infty} \mu(X_n) = \infty \implies \mu(\limsup_{n \to \infty} X_n) \geqslant \frac{1}{4c}
\]
On the other hand, with no constraint on $\mu(X_m \cap X_n)$ we have
\[
\sum_{n=1}^{\infty} \mu(X_n) < \infty \implies \mu(\limsup_{n \to \infty} X_n) = 0
\]
\end{lemma}

\begin{remark}\label{d=1}
The condition in \eqref{pairwise-ind} is called {\em pairwise almost independence}. To simplify the discussion henceforth, we shall require $d=1$ in our definition of pairwise almost independent.
\end{remark}

Now consider a measure space $(\Omega, \B)$ with two $\B$-measures $\ell$ and $\nu$. Suppose that there is a sequence of measurable sets $X_n$ such that the sets $X_n$ are pairwise almost independent for the measure $\ell$, and that for $n$ large enough, the measures satisfy $\ell(X_n) \approx 1/n$ and $\nu(X_n) \leqslant  \exp(-kn)$ for some positive constant $k>0$. Then a direct application of Lemma~\ref{BC1} shows that $\ell(\limsup X_n)>0$ and $\nu(\limsup X_n) = 0$. So the goal is to set up such a sequence of sets in our context.

In the construction in Section~\ref{Singular} however, we do not directly construct a sequence $X_n$ with properties as above. Instead, it turns out natural to construct a doubly indexed sequence of sets $Y^{(m)}_n$ that for different $m$, are pairwise almost independent for the Lebesgue measure $\ell$ i.e. when $m_1 \neq m_2$ they satisfy the inequality ~\eqref{pairwise-ind}. In addition, the sets $Y^{(m)}_n$ have the property that there are positive integers $N, j$ independent of $m$, such that $\ell(Y^{(m)}_n) \approx 1/n^j$ for $n>N$. For the measure $\nu$, there is a constant $k>0$ independent of $m$, such that $\nu(Y^{(m)}_n) \leqslant  \exp(-kn)$ for $n>N$.

Given such a doubly indexed sequence, the proposition below shows how to construct the sequence $X_n$, with the properties described above. The set $\limsup X_n$ is then a set with positive $\ell$ measure and zero $\nu$ measure.

\begin{proposition}\label{BC2}
Let $(\Omega, \B)$ be a probability space with measures $\ell$ and $\nu$. Let $Y^{(m)}_n$ be a doubly indexed sequence of measurable sets such that there exists a positive integer $N$ and constants $j \in \N$, $a \in \R$ with $a > 1$ and $r<1$ such that
\begin{equation}\label{e}
\frac{1}{a n^j} < \ell(Y^{(m)}_n) < \frac{a}{n^j}, \hspace{5mm} \nu(Y^{(m)}_n) < r^n
\end{equation}
for all $m$ and for all $n >N$. Further, suppose that the sets $Y^{(m)}_n$ are pairwise almost independent i.e. for all pairs $(Y^{(m_1)}_{n_1}, Y^{(m_2)}_{n_2})$ of sets with $m_1 \neq m_2$, there exists a constant $c>1$ such that \[
\ell(Y^{(m_1)}_{n_1} \cap Y^{(m_2)}_{n_2}) < c \cdot \ell(Y^{(m_1)}_{n_1}) \ell(Y^{(m_2)}_{n_2})
\]
Then there is a set $X$ such that $\ell(X)>0$ and $\nu(X) = 0$.
\end{proposition}

\begin{proof}
Define a sequence of sets $X_n$ as follows: for $n > N$, let $s(n) = \sum_{i=1}^{n} i^{j-1}$ and $t(n) = \sum_{i=1}^{n+1} i^{j-1}$ and define
\[
X_n = \bigcup_{m=s(n)}^{m=t(n)} Y^{(m)}_n
\]
Then for the measure $\ell$ we get
\[
\sum_{m=s(n)}^{m=t(n)} \ell(Y^{(m)}_n) - \sum_{s(n) \leqslant  m_1 < m_2 \leqslant  t(n)} \ell ( Y^{(m_1)}_n \cap Y^{(m_2)}_n) \leqslant  \ell(X_n) \leqslant  \sum_{m=s(n)}^{m=t(n)} \ell(Y^{(m)}_n)
\]
By the property of almost independence and the upper bound in Equation~\eqref{e}
\[
\ell(Y^{(m_1)}_n \cap Y^{(m_2)}_n) < c \cdot \ell( Y^{(m_1)}_n) \ell(Y^{(m_2)}_n) < \frac{ca^2}{n^{2j}}
\]
Hence, by the using the bounds in Equation~\eqref{e}
\[
\sum_{m=s(n)}^{m=t(n)} \frac{1}{an^j} - \sum_{m_1<m_2} \frac{ca^2}{n^{2j}} < \ell(X_n) < \sum_{m=s(n)}^{m=t(n)} \frac{a}{n^j}
\]
which reduces to
\[
\frac{1}{an} - \frac{ca^2}{n^2} < \ell(X_n) \leqslant  \frac{a}{n}
\]
This shows that there is positive integer $N$ and a constant $A> a$ such that for all $n>N$
\[
\frac{1}{An} < \ell(X_n) < \frac{A}{n}
\]
Now we show that the sets $X_n$ are almost independent. For $n_1 < n_2$ notice that
\[
X_{n_1} \cap X_{n_2} \subseteq \bigcup_{\substack{s(n_1) \leqslant  p \leqslant  t(n_1) \\ s(n_2) \leqslant  q \leqslant  t(n_2)}} Y^{(p)}_{n_1} \cap Y^{(q)}_{n_2}
\]
As a consequence, with $s(n_1) \leqslant  p \leqslant  t(n_1)$ and $s(n_2) \leqslant  q \leqslant  t(n_2)$,
\[
\ell(X_{n_1} \cap X_{n_2}) \leqslant  \ell\left(\bigcup_{p,q} Y^{(p)}_{n_1} \cap Y^{(q)}_{n_2}\right) \leqslant  \sum_{p,q} \ell(Y^{(p)}_{n_1} \cap Y^{(q)}_{n_2})
\]
Using pairwise almost independence of the $Y$'s, we get
\[
\ell(X_{n_1} \cap X_{n_2}) < \sum_{p,q} c \cdot \ell(Y^{(p)}_{n_1})\ell(Y^{(q)}_{n_2}) < \sum_{p,q} \frac{ca^2}{n_1^j n_2^j} = \frac{ca^2}{n_1n_2} < \frac{ca^2}{A^2} \ell(X_{n_1})\ell(X_{n_2})
\]
showing pairwise almost independence of the sequence $X_n$ with respect to the measure $\ell$.

For the measure $\nu$ note that
\[
\nu(X_n) \leqslant  \sum_{m=s(n)}^{t(n)} \nu(Y^{(m)}_n) < \sum_{m=s(n)}^{t(n)} r^n = n^{j-1} r^n
\]
The above inequality implies that for $N$ large enough, there is a constant $\rho <1$ such that for all $n>N$, we have $\nu(X_n) < \rho^n$. Finally by an application of Lemma~\ref{BC1}, we have $\ell(\limsup X_n)>0$ and $\nu(\limsup X_n)=0$.
\end{proof}

\section{The $PSL(2,\Z)$ example} \label{Torus}

We shall first explain the construction of the singular set in the special case when the group is $PSL(2,\Z)$. The group $PSL(2,\Z)$ is the mapping class group of the torus. The \teichmuller space of the torus is $\H^2$ and the group $PSL(2,\Z)$ acts on it by fractional linear transformations. The projective class of a measure foliation on the torus is determined by its slope, and can be marked as a point on the boundary $\partial \H^2 = \R \cup \{\infty\}$. The simple curves on the torus are the rational points. The Farey graph is constructed with the rational points as the vertex set with an edge between two rational points if the simple curves representing those points can be isotoped to intersect minimally i.e. in a single point. The group $PSL(2,\Z)$ acts on the Farey graph and it follows from the action that it is quasi-isometric to the trivalent tree dual to Farey graph. A part of the Farey graph and its dual tree are shown in Figure~\ref{dual}.

\begin{figure}[htb]
\begin{center}
\ \psfig{file=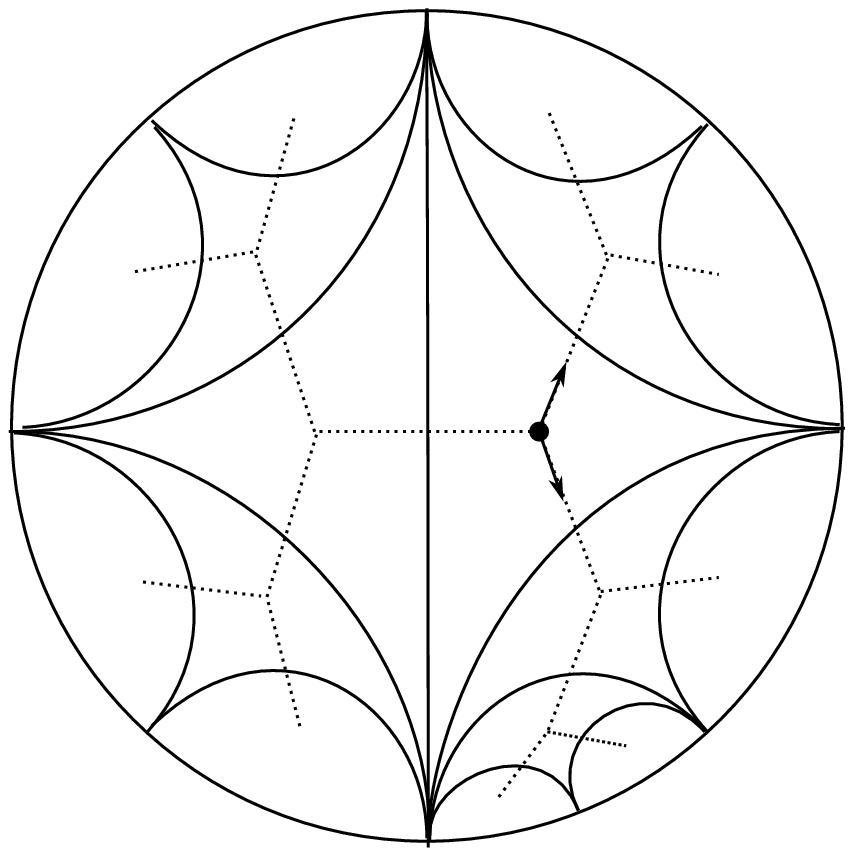, height=2truein, width=2truein} \caption{Farey graph and the dual tree} \label{dual}
\end{center}
\setlength{\unitlength}{1in}
\begin{picture}(0,0)(0,0)
\put(0,0.4){0}\put(1.1,1.55){1}\put(0,2.7){$\infty$}\put(-1.15,1.55){-1}\put(0.4,1.5){$\bb$}
\end{picture}
\end{figure}

To illustrate the main point of the construction, first restrict to the nearest neighbor non-backtracking random walk on the trivalent tree starting from the base vertex $\bb$ as shown in Figure~\ref{dual}. Moving forward in the tree from the base vertex, we can choose to move either ``right'' or ``left'' as shown by the arrows in the figure. The $PSL(2,\Z)$ generators that move the base vertex forward by right and left can be taken to be
\begin{equation}\label{PSL2Z}
R = \left[ \begin{array}{cc} 1 & 0\\ 1 & 1 \end{array} \right] \hskip 10pt , \hskip 10pt L = \left[ \begin{array}{cc} 1 & 1 \\ 0 & 1 \end{array} \right]
\end{equation}
respectively. Now consider the interval $[0,1]$ in $\partial \H = \R \cup \{ \infty \}$. Every irrational number $\x$ in $[0,1]$ corresponds to a unique non-backtracking sample path in the tree. This sample path can be written down as an infinite word $R^{a_1} L^{a_2} \cdots$, where all $a_n$ are positive integers. In fact, it is easy to see that the numbers $a_n$ are the coefficients in the continued fraction expansion of $\x$. The set of points in $[0,1]$ whose expansion begins with $R^n$ is $[0,1/n)$, whose Lebesgue measure is $1/n$. On the other hand, from the point of view of the nearest neighbor non-backtracking random walk, we are choosing $R$ from the possible choices $\{R, L\}$, $n$ times. So the set of infinite sample paths that begin with $R^n$ have probability $1/2^n$. This indicates that near parabolic fixed points the Lebesgue and the harmonic measure scale differently under a repeated application of that parabolic element. Later, we exploit this discrepancy to show that the measures are mutually singular.

To give the explicit construction of the singular set restricted to $[0,1]$, we shall first interpret the continued fraction expansion of $\x$ as a splitting sequence of a classical interval exchange on 2 bands.

\subsection{Classical interval exchanges:}
In a classical interval exchange map, an interval $I$ is partitioned into $d$ subintervals, these subintervals are permuted and then glued back preserving their orientation, to get $I$. The result is a Lebesgue measure preserving map from $I$ to itself. The data that completely determines a classical interval exchange map is: first, the lengths of the subintervals and second, the permutation used for gluing. To work with projective classes of foliations, we normalize the length of the base interval $I$ to be 1.

\begin{figure}[htb]
\begin{center}
\ \psfig{file=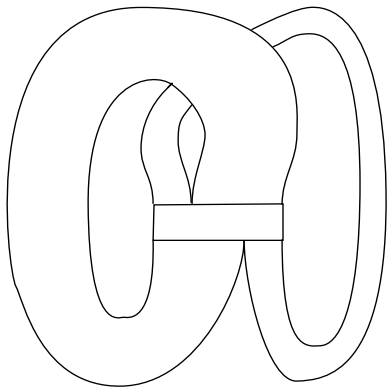, height=2truein, width=2truein} \caption{Classical Interval Exchange} \label{ciem}
\end{center}
\setlength{\unitlength}{1in}
\begin{picture}(0,0)(0,0)
\end{picture}
\end{figure}

A classical interval exchange map can be represented pictorially as follows: In the plane, draw the interval $I = [0,1]$ along the horizontal axis and then thicken it slightly in the vertical direction to get two copies, $I \cross (\delta)$ and $I \cross (-\delta)$. Call them top interval and bottom interval respectively. Subdivide the bottom interval into $d$ sub-intervals with lengths $\lambda_1, \lambda_2, \dotsc, \lambda_d$ from left to right. Subdivide the top interval into $d$ sub-intervals with lengths $\lambda_{\pi^{-1}(1)}, \lambda_{\pi^{-1}(2)}, \dotsc, \lambda_{\pi^{-1}(d)}$ from left to right. Now join each subinterval on the bottom to the corresponding subinterval on the top by a band of uniform width $\lambda_\alpha$. To determine the image of a point $t \in I$ under the interval exchange map, pick the subinterval on the bottom in which $t$ lies and flow $t$ along the band to the top. The only ambiguity in the definition occurs at the common endpoints of adjacent subintervals. This is removed by requiring the endpoint flow along the band that lies to the right.

Here, we shall work with {\em labeled} interval exchanges. A labeling is a bijection from the set $S = \{ 1, \cdots, d\}$ to the set of bands. Interval exchanges with the same exchange combinatorics but different labeling shall be regarded as different.

\subsubsection{Rauzy induction:}
We shall now describe an induction process on the space of interval exchanges called Rauzy induction. Call the positions on the top and bottom that are rightmost on the intervals the {\em critical} positions. Let $\alpha_0$ and $\alpha_1$ be the labels of the bands in the critical positions with $\alpha_1$ on the top. First, suppose that $\lambda_{\alpha_1} > \lambda_{\alpha_0}$. Then we slice as shown in Figure~\ref{rauzyinduct} till we hit the original interval for the first time.

\begin{figure}[htb]
\begin{center}
\ \psfig{file=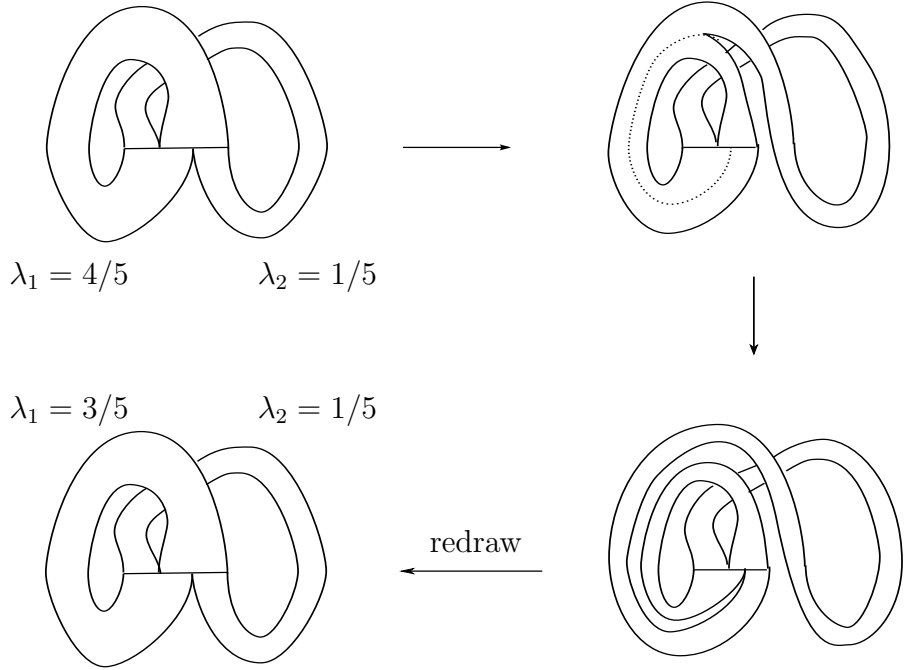, height=3.5truein, width=4.5truein} \caption{Rauzy Induction} \label{rauzyinduct}
\end{center}
\setlength{\unitlength}{1in}
\begin{picture}(0,0)(0,0)
\put(-2.4,2.6){$\lambda_1 = 4/5$} \put(-1.1,2.6){$\lambda_2 = 1/5$}\put(-0.2,1.2){redraw}\put(-2.4,1.9){$\lambda_1 = 3/5$}\put(-1.1,1.9){$\lambda_2 = 1/5$}
\end{picture}
\end{figure}

The $\alpha_1$ band remains in its critical position, but typically a band with a different label $\alpha'_0$ moves into the other critical position. Furthermore, the new width of $\alpha_1$ is $ \lambda_{\alpha_1} - \lambda_{\alpha_0}$. All other widths remain unchanged. If $\lambda_{\alpha_1} < \lambda_{\alpha_0}$ instead, we slice in the opposite direction, which in Figure~\ref{rauzyinduct} would be the analogous operation after flipping the picture about the horizontal axis. In either case, we get a new interval exchange with combinatorics and widths as described above. The operation we just described is called {\em Rauzy} induction. It is the same as the first return map under $T$ to the interval $I' = [0,\sum_{\alpha \neq \alpha_0} \lambda_{\alpha})$ in the first instance and $I' = [0, \sum_{\alpha \neq \alpha_1} \lambda_{\alpha})$ in the second. The interval exchange we get by the induction carries an {\em induced} labeling. Since Rauzy induction is represented pictorially by one band being split by another, it is also called a {\em split}. Iterations of Rauzy induction are called {\em splitting sequences}.

\subsubsection{Rauzy diagram:}
For classical interval exchanges with $d$ bands, construct an oriented graph $\overline{\cG}$ as follows: the nodes of the graph are combinatorial types $\phi$ of labeled classical interval exchanges with $d$ bands. The initial labeling induces a labeling of the interval exchanges obtained by Rauzy induction. So we draw an arrow from $\phi$ to $\phi'$, if $\phi'$ is a labeled combinatorial type resulting from splitting $\phi$. For each node $\phi$, there are exactly two arrows coming out of it. A splitting sequence gives us a directed path in $\overline{\cG}$. A finite splitting sequence starting from $\phi$ shall be called a {\em stage} in the expansion starting from $\phi$.

\subsubsection{Encoding Rauzy induction on the parameter space:}
A choice of labeling gives a bijection between the set of bands and the standard basis of $\R^d$. Let $\{e_\alpha \}$ denote the standard basis under the bijection. We get a map from the set of interval exchanges sharing the same combinatorics, into $(\R_{\geqslant 0})^d \setminus \{ 0 \}$ by thinking of $(\lambda_\alpha)$ as co-ordinates for $e_\alpha$. Because of the normalization $\sum \lambda_\alpha =1$, the image is the standard $(d-1)$-simplex $\Delta$.

Each instance of the Rauzy induction can be encoded as a non-negative matrix as follows: Let $M_{\alpha \beta}$ be the $d \times d$-matrix with the $(\alpha,\beta)$ entry 1 and all other entries 0. If the labels of the bands in the critical positions are $\alpha_0$ and $\alpha_1$, then the relationship between the old and new width data can be expressed by
\begin{equation}\label{width-rel}
\lambda = E \lambda'
\end{equation}
where the matrix $E$ has the form $E = I + M$. In the first instance of the split, when $\lambda_{\alpha_1} > \lambda_{\alpha_0}$, the matrix $M = M_{\alpha_1\alpha_0}$; in the second instance of the split, when $\lambda_{\alpha_0} > \lambda_{\alpha_1}$, the matrix $M = M_{\alpha_0 \alpha_1}$. Thus, in either case the matrix $E$ is an elementary matrix, in particular $E \in SL(d; \Z)$.

Proceeding iteratively, we associate a matrix in $SL(d, \Z)$ to any finite splitting sequence by requiring that the matrix $Q$ at any stage in the sequence is obtained by multiplying the matrix $Q'$ for the preceding stage on the right by the elementary matrix $E$ associated to that particular split i.e. $Q = Q'E$. In this way, starting with an interval exchange determined by the width data $(\lambda_\alpha)$, we get an expansion by repeated splitting.

The Rauzy induction is undefined when $\lambda_{\alpha_0} = \lambda_{\alpha_1}$. However, for any labeling and combinatorial type, this always gives a co-dimension 1 subset of $\Delta$ and hence has measure zero. The set of widths whose Rauzy induction stops in finite number of steps is a countable union of such sets, and hence of measure zero. Thus for almost every width data, one can associate an infinite expansion.

Fix an initial node $\phi_0$ and use the labeling to identify the space of interval exchanges at $\phi_0$ by their widths with $\Delta$. Expansions of the points $(\lambda_\alpha)$ in $\Delta$ by repeated Rauzy induction correspond to directed paths in $\overline{\cG}$ starting from $\phi_0$.

Given a matrix $Q$ with non-negative entries, we define the projectivization $JQ$ as a map from $\Delta$ to itself by
\[
JQ(\y) = \frac{Q \y}{\vert Q \y \vert}
\]
where if $\y = (y_1, y_2, \cdots, y_d)$ in co-ordinates then $\vert \y \vert = \sum \vert y_i \vert$.

The subset of widths in $\Delta$ whose expansion begins with some finite splitting sequence $\jmath$ is given by $JQ_\jmath(\Delta)$. To get estimates of relative probabilities of particular splitting sequences after $\jmath$ it becomes essential to consider the Jacobian $\J_\Delta (JQ_\jmath)$ of the projective linear map $JQ_\jmath$ from $\Delta$ to itself. It is known that \cite{Buf}
\[
\J_\Delta(JQ_\jmath)(\y) = \frac{1}{|Q_\jmath \y|^d}
\]
For a matrix $Q_\jmath$ of a stage $\jmath$, let $Q_\jmath(\alpha)$ be the column of $Q_\jmath$ corresponding to the band $\alpha$.
\begin{definition}
For a constant $C>1$, a stage $\jmath$ in the expansion is said $C$-distributed if the matrix $Q_\jmath$ of the stage has the property that for any pair of columns $Q_\jmath(\alpha)$ and $Q_\jmath(\beta)$
\[
\frac{1}{C} < \frac{\vert Q_\jmath(\alpha) \vert}{\vert Q_\jmath(\beta) \vert} < C
\]
\end{definition}

\noindent Suppose the splitting sequences $\jmath$ and $\kappa$, thought of as directed paths in $\overline{\cG}$, can be concatenated i.e. the sequence $\jmath$ can be followed by $\kappa$. We denote by $\jmath \ast \kappa$ the splitting sequence given by the concatenation. The main point of $C$-distribution is
\begin{lemma}\label{control}
If a stage $\jmath$ is $C$-distributed, then there exists a constant $c>1$ that depends only on $C$ and $d$, such that the relative probability that any sequence $\kappa$ follows $\jmath$ satisfies
\begin{equation}
\frac{1}{c} \ell(JQ_\kappa(\Delta)) < \frac{\ell(JQ_{\jmath \ast \kappa}(\Delta))}{\ell(JQ_\jmath(\Delta))} < c \cdot \ell(JQ_\kappa(\Delta))
\end{equation}
\end{lemma}
In other words, Lemma~\ref{control} says that if a stage $\jmath$ is $C$-distributed, the relative probability of a splitting sequence $\kappa$ following $\jmath$ is up to a constant that depends on $C$ and $d$ alone, the same as the probability that an expansion begins with $\kappa$. See \cite{Gad}

\begin{proof}
Let $Y_\jmath = JQ_\jmath(\Delta), Y_\kappa = JQ_\kappa(\Delta)$ and $Y_{\jmath \ast \kappa}= JQ_{\jmath \ast \kappa}(\Delta)$. Then
\begin{equation*}
\frac{\ell(Y_{\jmath \ast \kappa})}{\ell(Y_\jmath)} = \frac{\int_{Y_\kappa} \J_\Delta(JQ_\jmath)(\x) d \ell(\x)}{\int_{\Delta}\J_\Delta(JQ_\jmath)(\x) d \ell(\x)}
\end{equation*}
Because of the $C$-distribution of $\jmath$, the Jacobian $\J_\Delta(JQ_\jmath)$ at any point inside $Y_\kappa$ differs from the Jacobian $\J_\Delta(JQ_\jmath)$ at any point in $\Delta \setminus Y_\kappa$ by a factor that lies in $(1/C,C)$. This finishes the proof.
\end{proof}

\subsection{Classical interval exchange with 2 bands:}
Since there is exactly a single permutation over two letters, there is only two combinatorial types for a classical interval exchange with 2 bands.
\begin{figure}[htb]
\begin{center}
\ \psfig{file=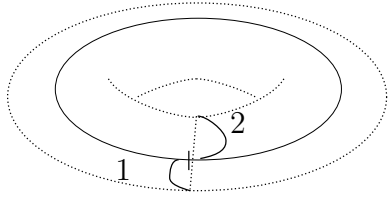, height=1truein, width=2truein} \caption{Interval exchange on the torus} \label{torustrack}
\end{center}
\setlength{\unitlength}{1in}
\begin{picture}(0,0)(0,0)
\put(-0.4,0.65){1} \put(0.2,0.9){2}
\end{picture}
\end{figure}

One of the ways to embed this interval exchange on the torus is shown in Figure~\ref{torustrack}. The transverse arc represents the base interval. For this embedding, the widths $(\lambda_1, \lambda_2)$ of the bands uniquely determines the foliation with slope $\lambda_2/\lambda_1$ in the standard homology basis of the torus. Thus the set $[0,1] \subset \R \cup \infty$ that we are interested in, is equivalent to $\lambda_2 < \lambda_1$ initially.

Now we consider expansions by repeated Rauzy inductions. At any stage, there are exactly two types of splits that are possible. When band 2 splits band 1 we denote it by $R$, and if vice versa we denote it by $L$. Since initially $\lambda_2 < \lambda_1$, all our expansions begin with $R$. The elementary matrices associated to the splits are exactly the same as in Equation~\eqref{PSL2Z}. Thus, the two moves possible on the dual tree correspond to the two possible splits of our classical interval exchange.

It is easy to check that the expansion stops in finite time if and only if the slope is rational. Thus an irrational slope gives an infinite splitting sequence. If we write down the infinite sequence as an infinite word $R^{a_1} L^{a_2} \cdots$, for positive integers $a_n$, then it is easy to check that the $a_n$ are exactly the coefficients in the continued fraction expansion of the slope. Thus we have encoded the infinite sample paths $R^{a_1} L^{a_2} \cdots$ for the nearest neighbor non-backtracking random walk on the trivalent tree as splitting sequences of the interval exchange.

The main point is
\begin{proposition}\label{eq-dist-torus}
In every infinite expansion $R^{a_1} L^{a_2} \cdots$, every stage at which one switches between the letters $R$ and $L$ or vice versa, is a $2$-distributed stage i.e. every stage of the form $R^{a_1} \cdots R^{a_{2k-1}} L$ or $R^{a_1} \cdots L^{a_{2k}} R$ is 2-distributed. In particular, almost every expansion becomes $2$-distributed infinitely often.
\end{proposition}
\noindent The proof of the proposition is left to the reader.

\subsection{The singular set for $PSL(2,\Z)$:}

Let $m$ be an odd integer. Consider the sets
\[
X_{m,n} = \{ \x \in (0,1): a_m(\x) > n \}
\]
By direct computation, the Lebesgue measure of $X_{1,n}$ is
\[
\ell(X_{1,n}) = \frac{1}{n+1}
\]
By Proposition~\ref{eq-dist-torus}, stages of the form $R^{a_1} \cdots L^{a_{2k}} R$ are $2$-distributed. The expansions in which a stage of this form does not occur, terminate in finite time, and hence are a subset of the rational numbers. This means that a stage of the form $R^{a_1} \cdots L^{a_{2k}} R$ occurs in almost every expansion. By Lemma~\ref{control}, the relative probability of such a stage being followed by the sequence $R^n$, is up to a universal constant the same as the probability that an expansion begins with $R^n$. Thus, for odd $m$,
\begin{equation}\label{rel-prob}
\ell(X_{m,n}) \approx \ell(X_{1,n-1})
\end{equation}
Consider the set
\[
X = \limsup_{n \text{ odd}} X_{n,n} = \bigcap_{k=1}^{\infty} \hskip 2pt \bigcup_{n \geqslant 2k-1} X_{n,n}
\]
that is, $X$ is the set of elements $\x \in (0,1)$ such that $\x \in X_{n,n}$ for infinitely many odd $n \in \N$. The restriction that $n$ be odd is not essential. To keep the explanations simple, we want the 2-distributed stages considered, to have the form $R^{a_1}L^{a_2} \cdots L^{a_{m-1}} R$, and not $R^{a_1} \cdots R^{a_{2k-1}} L$.

Let $\nu$ be the harmonic measure on $(0,1)$ for the nearest neighbor non-backtracking random walk on the trivalent tree. Since there is no backtracking involved, every infinite sample path in the trivalent tree is convergent. Moreover, almost every infinite sample path necessarily passes through some vertex of the form $R^{a_1}L^{a_2} \cdots L^{a_{n-1}} R (\bb)$ exactly once. So to compute the proportion of sample paths that converge into $X_{n,n}$ we can condition on the set of vertices represented by the initial sample paths $R^{a_1}L^{a_2} \cdots L^{a_{n-1}} R$. From each such vertex, the probability of converging into $X_{n,n}$ is exactly $1/2^n$. Hence $\nu(X_{n,n}) = 1/2^n$
\begin{proposition}\label{psl2z-set}
\[
\ell(X) > 0  \hskip 8pt , \hskip 8pt \nu(X) = 0
\]
\end{proposition}
\begin{proof}
By Equation~\eqref{rel-prob}, $\ell(X_{n,n}) \approx \ell(X_{1,n-1}) = 1/n$. The sets $X_{n,n}$ are not independent for the Lebesgue measure, but as we shall see in Claim~\ref{claim1}, they are pairwise almost independent~\eqref{pairwise-ind}. For the harmonic measure, the sets $X_{n,n}$ simply satisfy
\begin{equation}\label{harmonic-sum}
\sum \nu(X_{n,n}) < \infty
\end{equation}
The proposition then follows directly from the generalized Borel-Cantelli Lemma~\ref{BC1}.
\begin{claim}\label{claim1}
There exists some constant $c>1$ such that for any $n,m$ odd
\[
\ell \bigl(X_{m,m}\cap X_{n,n}\bigr) <  c  \cdot \ell \bigl(X_{m,m}\bigr) \ell\bigl(X_{n,n}\bigr)
\]
\end{claim}
\begin{proof}
Let $U = X_{m,m}\cap X_{n,n}$ and $V= X_{1,m-1} \cap X_{n-m+1, n}$. For each set of positive integers $S=\{a_1, \cdots, a_{m-1}\}$, let $W(S)$ be the set of points whose expansion begins with the sequence $R^{a_1}L^{a_2} \cdots L^{a_{m-1}} R$. By Proposition~\ref{eq-dist-torus}, the sets $W(S)$ partition a set of full measure. So it is enough to show that for all $S$
\[
\ell(W(S) \cap U) < c \cdot \ell(W(S) \cap X_{m,m}) \ell(W(S) \cap X_{n,n})
\]
Taking union over all $W(S)$ gives us the claim. Because $R^{a_1}L^{a_2} \cdots L^{a_{m-1}} R$ is a $2$-distributed stage, by Equation~\eqref{rel-prob}, up to a universal constant, the relative probabilities satisfy
\begin{eqnarray*}
P(X_{m,m} \vert W(S)) &\approx& \ell(X_{1,m-1}) \\
P(X_{n,n}\vert W(S)) &\approx& \ell(X_{n-m+1,n})\\
P(U \vert W(S)) &\approx& \ell(V)
\end{eqnarray*}
Thus it is enough to prove that there exists a constant $c>1$ such that for any $n, m$ odd and $n>m$
\begin{equation}\label{E2}
\ell(X_{1,m-1} \cap X_{n-m+1,n}) < c \cdot \ell(X_{1,m-1}) \ell(X_{n-m+1,n})
\end{equation}
In fact, we shall show that there exists a constant $c>1$ such that for any positive integers $m,k>1$ and any odd integer $n$
\begin{equation}\label{E3}
\ell(X_{1,m} \cap X_{n,k}) < c \cdot \ell(X_{1,m}) \ell(X_{n,k})
\end{equation}
Replacing $m$ by $(m-1)$, $n$ by $(n-m+1)$ and $k$ by $n$ in the above inequality implies Inequality~\eqref{E2}. The proof of Inequality~\eqref{E3} is as follows: For each set of positive integers $S'= \{a_1, \cdots, a_{n-1}\}$, let $W(S')$ be the set of points whose expansion begins with the sequence $R^{a_1}L^{a_2} \cdots L^{a_{n-1}} R$. By Proposition~\ref{eq-dist-torus}, the stages $R^{a_1}L^{a_2} \cdots L^{a_{n-1}} R$ are 2-distributed and the sets $W(S')$ partition a set of full measure. Moreover each such stage $W(S')$ either belongs entirely to $X_{1,m}$ or belongs entirely to the complement of $X_{1,m}$ depending on whether $a_1 > m$ or not. Because of $2$-distribution, we have $P(X_{n,k} \vert W(S')) \approx \ell(X_{1,k-1})$. So
\begin{eqnarray*}
\ell(X_{1,m} \cap X_{n,k}) &\approx& \sum_{W(S') \subset X_{1,m}} P(X_{n,k} \vert W(S')) \ell(W(S')) \\
&\approx& \ell(X_{1,k-1}) \sum_{W(S') \subset X_{1,m}} \ell(W(S')) \\
&\approx& \ell(X_{1,k-1}) \ell(X_{1,m})\\
\end{eqnarray*}
Finally, by Equation~\eqref{rel-prob}, we have $\ell(X_{1,k-1}) \approx \ell(X_{n,k})$. Using this in the above equation finishes the proof of Inequality~\eqref{E3}, and hence of the claim.
\end{proof}
\end{proof}
\noindent Finally to get a singular set of full measure, we consider the union
\[
Y = \bigcup_{g \in PSL(2,\Z)} g X
\]
The set $Y$ is a countable union of translates of $X$, so $\nu(Y) = 0$. On the other hand, $Y$ is a set invariant under the action of $PSL(2,\Z)$. By the ergodicity of the action of $PSL(2,\Z)$ on $\partial \H^2$, invariant sets have zero or full measure. Hence $\ell(Y)=1$ which proves that $\ell$ and $\nu$ are mutually singular.

\subsection{General random walks on $PSL(2, \Z)$:}
The appropriate notion of a non-elementary subgroup of $PSL(2,\Z)$ is that the subgroup contains a pair of hyperbolic isometries $h_1$ and $h_2$ of $\H^2$ with distinct attracting and repelling fixed points in $\partial \H^2$. Since we do not assume that the initial distribution on $PSL(2, \Z)$ is symmetric i.e. $\mu(g) = \mu(g^{-1})$, we assume that the semi-group generated by the support contains the hyperbolic isometries $h_1$ and $h_2$ as above. By conjugating the semigroup by a rotation of $\D^2$, we can assume that the attracting fixed point of $h_1$ is in $(0,1)$. For such a finitely supported initial distribution, the same construction produces a singular set provided Equation~\ref{harmonic-sum} still holds. So it is enough to show that for $n$ large enough, the harmonic measure of $X_{n,n}$ decays exponentially i.e. there exists a constant $r < 1$ and a positive integer $N$ such that $\nu(X_{n,n}) < r^n$ for all $n > N$.

Since the group $PSL(2,\Z)$ is quasi-isometric to the trivalent tree the random walk can be projected on to the trivalent tree by using the quasi-isometry. We assume that the quasi-isometry maps the identity $1 \in PSL(2,\Z)$ to the base-point $\bb$, so that the projected random walk starts from $\bb$. Because of the quasi-isometry, the random walk on $PSL(2,\Z)$ satisfies estimates similar to those satisfied by the projected random walk. The actual constants in the estimates for $PSL(2, \Z)$ also depend on the quasi-isometry constants, but that does not affect the construction. Hence, it is enough to prove an estimate of the above form holds for the projected random walk on the trivalent tree.

The estimate for $X_{n,n}$ can be proved by adapting the proof of Lemma 5.4 of~\cite{Mah2}. For a vertex $v$ distinct from $\bb$, let $C_v$ be the ``cone'' of vertices in the tree such that the unique geodesic connecting them to $\bb$ passes through $v$. Similar to Proposition 5.3 of~\cite{Mah2}, one shows that there is a definite positive integer $K$ and a constant $\epsilon >0$ that depend only on the initial distribution, such that if $d(\bb,v) = K$, then $\nu(\overline{C_v}) \leqslant 1-\epsilon$. To prove that $\nu(\overline{C_v})$ can be bounded away from 1, we construct a pair of Schottky cones for suitable powers of $h_1$ and $h_2$, similar to what Maher does at the beginning of Section 5 of~\cite{Mah2}. The bound for $\nu(\overline{C_v})$ then follows by exploiting the $\delta$-hyperbolicity of the tree. Finally, the set $X_{n,n}$ is the disjoint union of limit sets of cones for vertices of the form $v = R^{a_1}L^{a_2} \cdots L^{a_{n-1}}R^{n}$. So by using the estimate for the cones inductively we get the exponential decay for $\nu(X_{n,n})$.

\subsection{The general mapping class groups:}

The element $R$ as a split of the classical interval exchange on the torus in Fig~\ref{torustrack}, can be seen to be a Dehn twist in the vertex curve given by band 1. This vertex curve is the curve with slope $0$ in the Farey graph of Figure~\ref{dual}, and $R$ acts on $\partial \H^2$ as a parabolic element with fixed point $0$.

For the general mapping class groups, the action of the reducible elements on $\pmf$, in particular Dehn twists, is similar to the action of $R$ on $\partial \H^2$. We show that a suitable ``half-open'' neighborhood in $\pmf$ of the fixed point of a Dehn twist scales down at the rate 1/poly for the Lebesgue measure and $\exp(-n)$ for the harmonic measure, under the repeated application of that Dehn twist. So the construction of the singular set is essentially similar. However, because of the issues outlined below, it becomes technically harder in this context to construct the sequence of sets $Y^{(m)}_n$ to which Proposition~\ref{BC2} is applied, and trickier to prove the above mentioned measure estimates for these sets.

First, a natural class of charts on $\pmf$ is obtained from {\em complete} train tracks on the surface. Projective classes of measured foliations carried by a complete train track can be encoded by splitting sequences of that track. However, this involves choosing a large edge to split at each stage of the expansion, which means that this encoding is not unique. More importantly, the combinatorics of the corresponding Rauzy diagram i.e the directed graph with vertices as the combinatorial types of complete train tracks and arrows given by splits, is not well understood. As a result, we do not know how to prove an analog of Proposition~\ref{eq-dist-torus} for complete train tracks, and such an analog is essential to estimate Lebesgue measure. To circumvent these difficulties, we consider non-classical interval exchanges instead viz. complete train tracks with a single vertex. Here, the Rauzy diagram, is better understood \cite{Boi-Lan}. In particular, Boissy and Lanneau~\cite{Boi-Lan} characterize the attractors of this directed graph (alternatively known as the Rauzy classes) in terms of suitable irreducibility criteria for the combinatorics of the non-classical interval exchanges. This allows us to show that the combinatorics of the initial non-classical interval exchange that we construct, belongs to some attractor. The Rauzy induction which was defined for classical interval exchanges generalizes directly to non-classical interval exchanges. Unlike the situation for complete train tracks, for non-classical interval exchanges, the expansions given by repeated Rauzy induction uniquely encode the projective classes of measured foliations carried by them. Secondly, some of the techniques used to study expansions of classical interval exchanges, generalize to non-classical interval exchanges. In particular, in \cite{Gad}, we proved an analog of Proposition~\ref{eq-dist-torus} for non-classical interval exchanges, namely Theorem~\ref{Uniform-distortion}.

The second issue is that for the general mapping class groups, the appropriate quasi-isometric models are the marking complexes. Unlike $PSL(2,\Z)$ which is $\delta$-hyperbolic, the geometry of the marking complex is intricate. Instead, we project the mapping class group random walk to the curve complex. The combination of the results of Kaimanovich-Masur and Klarreich shows that one does not lose any stochastic information about the random walk in the process. There are a number of advantages in considering the curve complex: foremost is that the curve complex is $\delta$-hyperbolic; the coarse negative curvature proves to be very useful in analyzing the projected random walk. Secondly, in the curve complex, the set of essential simple closed curves carried by a non-classical interval exchange is quasi-convex. As we shall see, a recent technique due to Masur-Mosher-Schleimer can be utilized to get information about sub-surface projections to annuli that are carried deeply inside this quasi-convex set. This information enables us to apply the key tool, Theorem~\ref{exp-decay} of Maher, to get the required exponential decay estimate for the harmonic measure.

\section{Non-classical interval exchanges}\label{Niem}

A train track on the surface is an embedded 1-dimensional CW complex in which there is a common line of tangency to all the 1-dimensional branches that join at a 0-dimensional switch.  This splits the set of branches incident on a switch into two subsets, one of which will be called the set of outgoing edges and the other the set of incoming edges. We consider only {\em large} train tracks i.e. tracks whose complementary regions in the surface are either ideal polygons with at least 3 cusps or once-punctured polygons.

A {\em train route} is a regular smooth path in the train track. In particular, a train route traverses a switch only by passing from an incoming edge to an outgoing edge or vice versa. An essential simple closed curve is said to be {\em carried} by a train track if it is isotopic to a closed train route. A train track is {\em recurrent} if for every branch there is an essential simple closed curve carried by the track such that it passes through the branch. Similarly, a train track is {\em transversely recurrent} if for every branch there is an essential simple closed curve hitting the train track efficiently (i.e. the complement of the union of the train track and the curve has no bigons) such that it intersects the branch at least once. A track that is both recurrent and transversely recurrent is called {\em birecurrent}.

One can assign non-negative weights to the branches so that the sum of the weights of the outgoing branches at a switch is equal to the sum of the weights of the incoming branches. In particular, by counting the number of times a carried multi-curve passes over a branch we get an assignment of integral weights that automatically satisfies the switch conditions. After normalizing so that the sum of weights is 1, the set of possible weights on the train track can be identified with a set of projective classes of measured foliations. These projective classes of measured foliations are said to be {\em carried} by the train track.

A train track is {\em maximal} if it is not a proper subtrack of any other train track. A maximal birecurrent train track is called {\em complete}. It is easy to see that if a train track is complete then all of it's complementary regions are ideal triangles or once punctured monogons. The converse is not true. See \cite{Pen-Har} for examples. The set of normalized weights carried by a complete train track gives a chart of $\pmf$. A generic foliation carried by a complete train track has a three pronged singularity in each complementary ideal triangle and a simple zero in every once-punctured monogon.

For technical reasons mentioned at the end of Section~\ref{Torus}, instead of complete train tracks, we consider {\em non-classical interval exchanges}. A non-classical interval exchange is a complete train track with a single switch. This means that there is exactly one switch condition. Such a train track can be drawn as an interval with bands similar to the picture for classical interval exchanges. Hence the terminology non-classical interval exchange. The main difference here is that in a non-classical interval exchange, there necessarily exist bands that run from top to top and bottom to bottom. Such bands shall be called orientation reversing bands. The set of possible normalized widths that can be assigned to the bands can be identified with projective classes of measured foliations carried by the non-classical interval exchange. Similar to classical interval exchanges, Rauzy induction is defined for non-classical interval exchanges by comparing widths of the bands in critical positions on the extreme right of the base interval, and then splitting the broader band by the narrower one till we hit the base interval again. In general, splitting a recurrent train track need not give a recurrent train track; Rauzy induction for non-classical interval exchanges however preserves recurrence. Splitting a non-classical interval exchange in this manner gives a new non-classical interval exchange with widths given exactly as in Equation~\eqref{width-rel}. Moreover, if the initial non-classical interval exchange is labeled then the labeling induces forward under Rauzy induction.

Starting with any permissible initial widths, repeated Rauzy induction gives a unique expansion associated to it. The induction is not defined when $\lambda_\alpha = \lambda_\beta$. For non-classical interval exchanges, this is always a co-dimension 1 condition, and hence measure zero. The set of foliations for which the expansion stops after a finite number of steps is a countable union of such measure zero sets, and hence of measure zero. Thus a full measure set of foliations has infinite expansion.

\subsection{Parameter space of a non-classical interval exchange:}
For all non-classical interval exchanges $T$ sharing the same combinatorics $\phi$, a choice of labeling gives a bijection between the set of bands and the standard basis of $\R^d$. Label the standard basis $\{e_\alpha \}$ using the bijection. We get a map from the set of such $T$ into $(\R_{\geqslant 0})^d \setminus \{ 0 \}$ by thinking of the widths $(\lambda_\alpha)$ as co-ordinates for $e_\alpha$. Denote the image by $W$. The normalization $\sum \lambda_\alpha =1$, restricts $W$ to lie in the standard $(d-1)$-simplex $\Delta$. The other constraint that points in $W$ satisfy is the switch condition imposed by the combinatorics $\phi$. Let $S_t$ and $S_b$ be the labels of orientation reversing bands on the top and the bottom respectively. The switch condition is equivalent to
\[
\sum_{\alpha \in S_t} \lambda_{\alpha}  = \sum_{\alpha \in S_b} \lambda_{\alpha}
\]
Thus $W$ is the intersection with $\Delta$ of a codimension 1 subspace of $\R^d$. For $\alpha \in S_t$ and $\beta \in S_b$, let $e_{\alpha \beta}$ be the midpoint of the edge $[e_\alpha, e_\beta]$ of $\Delta$ joining the vertices $e_{\alpha}$ and $e_{\beta}$. The subset $W$ is the convex hull of the points $e_{\alpha \beta}$ and $e_{\gamma}$ for $\gamma \notin S_t \cup S_b$. Since there are finitely many choices for the pair $(S_t,S_b)$ of disjoint subsets, the set of possible convex codimension 1 subsets of $\Delta$ that could be $W$ is also finite. We call the different subsets of $\Delta$ coming from all possible pairs $(S_t,S_b)$, {\em configuration spaces}.

\subsection{Combinatorics of the Rauzy diagram:}
Similar to the Rauzy diagram for a classical interval exchanges, construct an oriented graph $\overline{\cG}$ for non-classical interval exchanges as follows:  the nodes of the graph are combinatorial types $\phi$ of labeled non-classical interval exchanges with $d$ bands. Draw an arrow from $\phi$ to $\phi'$, if $\phi'$ is a combinatorial type resulting from splitting $\phi$. For each node $\phi$, there are at most two arrows coming out of it. A splitting sequence gives us a directed path in $\overline{\cG}$.

It is also possible to construct a different oriented graph $\cG$ in which the nodes are combinatorial types of non-classical interval exchanges with $d$ bands {\em without} labeling. The arrows are drawn as before. The graph $\cG$ is analogous to the {\em reduced} Rauzy diagram for irreducible classical i.e.m. See Section 4.2 of Yoccoz \cite{Yoc}. There is an obvious map from $\overline{\cG}$ to $\cG$ given by forgetting the labeling.

However, there are some key differences between the Rauzy diagrams for classical and non-classical interval exchanges.
For irreducible classical interval exchanges, each component of the Rauzy diagram is {\em strongly connected} i.e.~any node can be joined to any other node by a directed path. This implies that the reduced Rauzy diagram is also strongly connected. This is not the case with non-classical interval exchanges.

In \cite{Boi-Lan}, Boissy and Lanneau show that for non-classical interval exchanges, every component of $\overline{\cG}$ has strongly connected pieces termed {\em attractors}. Each attractor corresponds to a connected component of the principle stratum. Additionally, Boissy and Lanneau introduce the notion of {\em geometric irreducibility} of a non-classical interval exchanges, and show that every geometrically irreducible non-classical interval exchange sits in some attractor of the Rauzy diagram.

In the next section, we explicitly construct initial combinatorics for a non-classical interval exchange on any surface. It can be directly checked from Definition 3.1 in \cite{Boi-Lan} that our construction gives us combinatorics that is geometrically irreducible. Hence, the initial combinatorics lies in some attractor of the Rauzy diagram.

\subsection{Dynamics:}
Assume that the initial combinatorics is what we construct in the next section, ensuring that we are in an attractor of the Rauzy diagram. We denote the initial combinatorics and labeling by $\phi_0$. Let $W_0$ be the configuration space at $\phi_0$ and let $\x$ in $W_0$ be a non-classical interval exchange. Identical to classical interval exchanges, each step of Rauzy induction is encoded as an elementary matrix $E$. The expansion of $\x$ by repeated Rauzy induction gives a directed path starting from $\phi_0$ in the Rauzy diagram $\overline{\cG}$. Iteratively, to each stage $\phi_0 \to \phi_{\x,1} \cdots \to \phi_{\x,n}$ in the expansion of $\x$, there is an associated matrix $Q_{\x,n}$ such that $Q_{\x,n} = Q_{\x,n-1}E$, where $E$ is the elementary matrix for the split $\phi_{\x,n-1} \to \phi_{\x,n}$. Let $W_n$ be the configuration space at $\phi_{\x,n}$. One also gets a sequence of points $\x^{(n)} \in W_n$ such that $JQ_{\x,n} (\x^{(n)})= \x$.

Suppose $\jmath: \phi_0 \to \cdots \to \phi$ is a finite splitting sequence with associated matrix $Q_\jmath$. The set of $\x$ in $W_0$ whose expansion begins with $\jmath$ is the set $JQ_\jmath(W)$, where $W$ is the configuration space at $\phi$. We call $\jmath$ a {\em stage} in the expansion.

As in the case of classical interval exchanges, we have to understand the probability that a stage $\jmath$ is followed by a particular splitting sequence. The difference here is that we have to consider the Jacobian of the restriction of the projective linear map to the configuration spaces i.e. the Jacobian of $JQ_\jmath: W \to W_0$. We denote this Jacobian by $\J(JQ_\jmath)$. Because of this, the probability that a particular split follows $\phi$ can be quantitatively very different from the naive estimate using the full Jacobian $\J_\Delta$ of $JQ_\jmath: \Delta \to \Delta$, which was used in the classical case. See Section 6 in \cite{Gad}. Nevertheless, in \cite{Gad}, we prove an analog of Proposition~\ref{eq-dist-torus} for non-classical interval exchanges. We state this below.

For a constant $C>1$, we define a $C$-uniformly distorted stage to be a stage $\jmath$ such that for any pair of distinct points $\y, \y'$ in $W$,
\[
\frac{1}{C} \leqslant  \frac{\J(JQ_\jmath)(\y)}{\J(JQ_\jmath)(\y')} \leqslant  C
\]
Suppose that the final combinatorics $\phi$ of $\jmath$ is the same as $\phi_0$, and let $\kappa$ be a finite splitting sequence starting from $\phi_0$. Let $W$ be the configuration space of the stage $\kappa$. Exactly similar to Lemma~\ref{control}, we have

\begin{lemma}\label{nciem-control}
If the stage $\jmath$ is $C$-uniformly distorted, then there exists a constant $c>1$ that depends only on $C$ and $d$, such that the relative probability that $\kappa$ follows $\jmath$ satisfies
\begin{equation}
\frac{1}{c} \ell(JQ_\kappa(W_0)) < \frac{\ell(JQ_{\jmath \ast \kappa}(W))}{\ell(JQ_\jmath(W_0))} < c \cdot \ell(JQ_\kappa(W))
\end{equation}
\end{lemma}
Thus, Lemma~\ref{nciem-control} gives us exactly the same control as Lemma~\ref{control} for estimating relative probabilities.

It is straightforward to check that if the matrix $Q_\jmath$ is $C$-distributed, then the stage $\jmath$ is $C$-uniformly distorted. The converse need not be true. However, it means that in the analog of Proposition~\ref{eq-dist-torus}, it suffices to show that starting from any stage, almost every expansion gets $C$-distributed. The precise statement proved as Theorem 1.3 in \cite{Gad}, is:

\begin{theorem}\label{Uniform-distortion}
Suppose $\jmath: \phi_0 \dotsc \to \phi$ is a stage in the expansion, with $W$ the configuration space at $\phi$ and $Q_\jmath$ the associated matrix. There exists a constant $C>1$, independent of $\jmath$, such that for almost every $\x \in JQ_\jmath(W)$, there is a future stage $\phi_{\x,m}$ after $\phi$, such that the stage $\phi_{\x,m}$ is $C$-distributed. Additionally, by choosing $C$ large enough, we can assume that $\phi_{\x,m}$ is combinatorially the same as $\phi_0$.
\end{theorem}

\section{Construction of combinatorics for the initial non-classical interval exchange}\label{Combinatorics}
When a splitting sequence of a labeled non-classical interval exchange gives back the same labeling and combinatorics, it corresponds to the action of a mapping class.

In Section~\ref{Torus}, the Dehn twist $R$ in $PSL(2,\Z)$ was realized as a split of the interval exchange in Figure~\ref{torustrack}. In this section, we show that it is possible to emulate this phenomena for any surface i.e. given a surface with genus $g$ and $m$ punctures, we construct a non-classical interval exchange on it such that a Dehn twist in one of it's vertex cycles is realized by a splitting sequence. We call this splitting sequence, the {\em Dehn twist} sequence.

\subsection{The construction:} The complementary regions to a non-classical interval exchange on a surface of genus $g$ and $m$ punctures consist of $(4g-4+m)$ ideal triangles and $n$ once punctured monogons. We claim that any interval with bands that has the right number of ideal triangles and once punctured monogons in the complement is a non-classical interval exchange on the surface. The maximality and recurrence of the underlying train track are immediate. Transverse recurrence follows from an easy application of Corollary 1.3.5 in \cite{Pen-Har}.

We break down our construction into several cases. Throughout, the non-classical interval that we construct contains a single orientation preserving band labeled $B$ whose end on top will be the leftmost on $I$ and whose other end on bottom will be rightmost on $I$. All other bands are orientation reversing. If we split the bottom end of $B$ twice by all bands on top, then the base interval of the resulting interval exchange becomes the top end of $B$ and the resulting labeling and combinatorics is identical to the initial. The associated mapping class is easily seen to be a Dehn twist in the curve given by going along $B$ from it's bottom end to it's top end and then back along $I$ from left to right till we get back to the bottom end of $B$. Thus, we have a Dehn twist sequence of the non-classical interval exchange. The other details of the construction differ according to the case. The cases do not include some surfaces of low complexity. We shall construct the non-classical interval exchanges for these at the end of the section.

Next, we describe some of the essential pictures we need in the construction.

\subsubsection{The basic block:}

Consider a picture of a horizontal interval $I$ with orientation reversing bands on one side as shown in Figure~\ref{basicblock}.

\begin{figure}[htb]
\begin{center}
\ \psfig{file=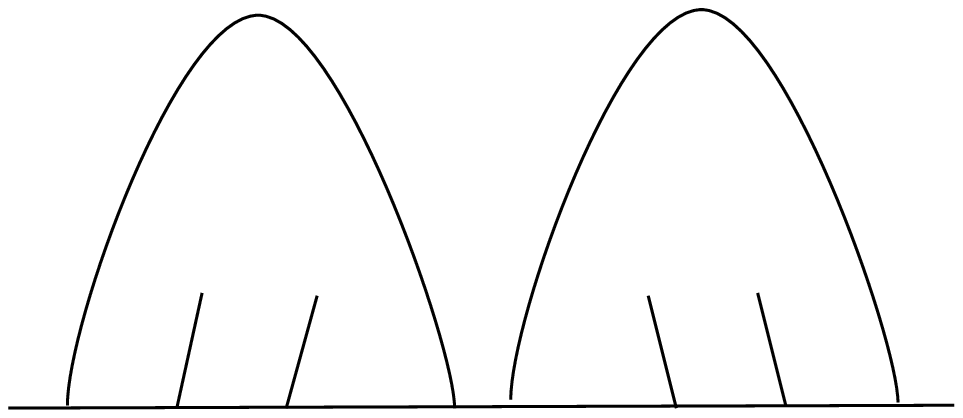, height=1.5truein, width=4truein} \caption{Basic Block} \label{basicblock}
\end{center}
\setlength{\unitlength}{1in}
\begin{picture}(0,0)(0,0)
\put(-1.3,2){1}\put(1.3,2){2}\put(-1.2,1.15){3}\put(1.15,1.15){4}\put(-0.7,1.15){4}\put(0.7,1.15){3}
\end{picture}
\end{figure}

The right end of band 1 is adjacent to the left end of band 2 along $I$. The pair of segments marked 3 and 4 are the ends of bands labeled 3 and 4 respectively. We have drawn only the ends of bands 3 and 4 to keep the figure simple. It is easy to check that the complementary regions of the basic block contains two ideal triangles: the first triangle has sides made up of the bands 1,3 and 4, and the second triangle has sides made up of bands 2,3 and 4. These triangles are on the ``inside'' of the bands 1 and 2 respectively.

\subsubsection{The outer blocks:}

We call the pictures in Figures~\ref{outerblock1} and~\ref{outerblock2} as {\em outer block 1} and {\em outer block 2} respectively. We shall use one or the other in the construction depending on the case in question.

\begin{figure}[htb]
\begin{center}
\ \psfig{file=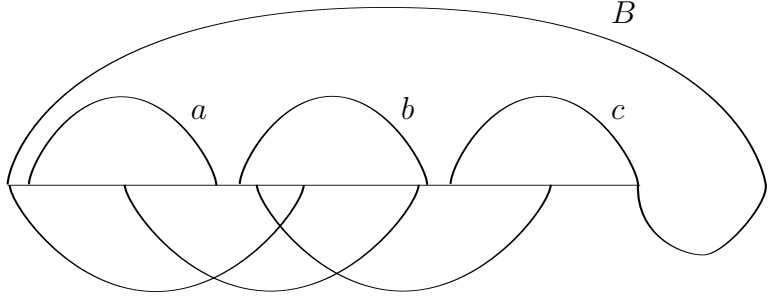, height=1.5truein, width=4truein} \caption{Outer Block 1} \label{outerblock1}
\end{center}
\setlength{\unitlength}{1in}
\begin{picture}(0,0)(0,0)
\put(1.2,2){$B$}\put(-1,1.5){$a$}\put(0.1,1.5){$b$}\put(1.2,1.5){$c$}
\end{picture}
\end{figure}

In outer block 1, we can identify 3 ideal triangles in the complement. The bands on the bottom of $I$ cut out two ideal triangles: the first contains the cusp that is leftmost on bottom and the second contains the cusp that is the rightmost on bottom. Additionally, the region bounded by the bands $a,b,c$ and $B$ is an ideal triangle.

\begin{figure}[htb]
\begin{center}
\ \psfig{file=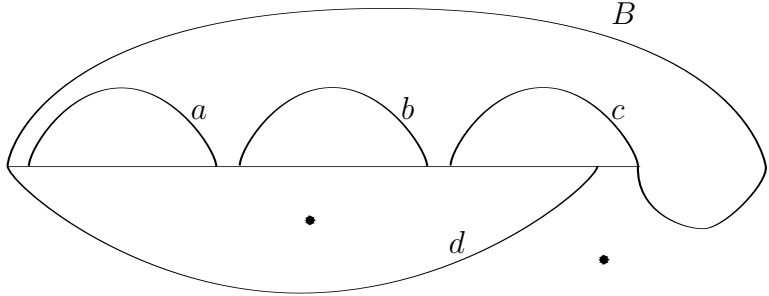, height=1.5truein, width=4truein} \caption{Outer Block 2} \label{outerblock2}
\end{center}
\setlength{\unitlength}{1in}
\begin{picture}(0,0)(0,0)
\put(1.2,2){$B$}\put(-1,1.5){$a$}\put(0.1,1.5){$b$}\put(1.2,1.5){$c$} \put(0.35,0.8){$d$}
\end{picture}
\end{figure}

In outer block 2, it is necessary to have one puncture inside and one puncture outside band $d$ as shown in Figure~\ref{outerblock2} to have those complementary regions as once punctured monogons.

Now we get to the various cases in the construction.

\subsubsection{Case 1: $m=0, g \geqslant 4$} Construct first the base interval $I$ along with outer block 1. In a separate picture, arrange $(g-1)$ basic blocks laid side to side, along the top of $I$. See Figure~\ref{stuffing}.

\begin{figure}[htb]
\begin{center}
\ \psfig{file=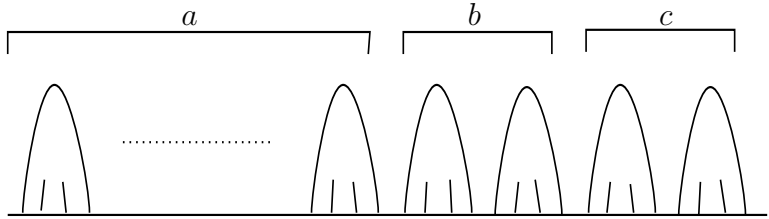, height=1truein, width=4truein} \caption{The stuffing} \label{stuffing}
\end{center}
\setlength{\unitlength}{1in}
\begin{picture}(0,0)(0,0)
\put(-1.05,1.6){$a$}\put(0.45,1.6){$b$}\put(1.45,1.6){$c$}
\end{picture}
\end{figure}

Take the basic blocks under the bracket $a$ in Figure~\ref{stuffing} and insert them inside the orientation reversing band $a$ in outer block 1. Similarly insert the basic block under brackets $b$ and $c$ inside the orientation reversing bands $b$ and $c$ respectively in outer block 1. In other words, Figure~\ref{stuffing} is super-imposed on Figure~\ref{outerblock1} such that the basic blocks under the marked brackets sit inside the corresponding orientation reversing bands.

Finally, notice that apart from the region on the inside of band $a$ and outside the basic blocks inserted inside $a$, all the complementary regions are ideal triangles. The exceptional region is an ideal polygon with $2(g-3)+1$ cusps. This can be sub-divided into $2(g-3)-1$ ideal triangles by adding in bands $a_i$ for $i = 1, \cdots, 2(g-3)-2$ as shown in Figure~\ref{add-bands}

\begin{figure}[htb]
\begin{center}
\ \psfig{file=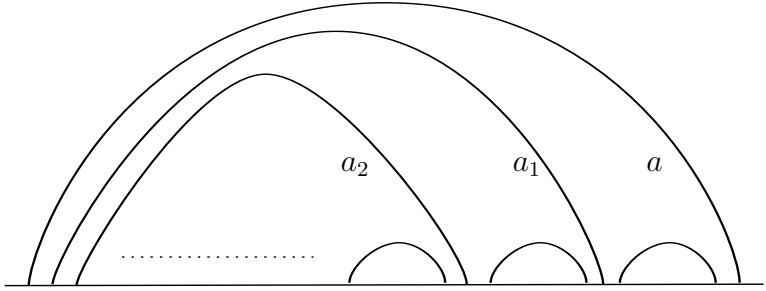, height=1.5truein, width=4truein} \caption{Dividing the polygon into ideal triangles}\label{add-bands}
\end{center}
\setlength{\unitlength}{1in}
\begin{picture}(0,0)(0,0)
\put(1.4,1.2){$a$}\put(0.7,1.2){$a_1$}\put(-0.2,1.2){$a_2$}
\end{picture}
\end{figure}

We claim that the resulting interval exchange is a non-classical interval exchange with a Dehn twist sequence, on a surface of genus $g$, with $g \geqslant 4$. First, it is easy to see that all complementary regions are ideal triangles. Second, the Dehn twist sequence is given by splitting $B$ twice by all the bands on top of $I$. So it is enough to show that there are $(4g-4)$ ideal triangles.

The number of ideal triangles inside the band $a$ is: $2(g-3)$ ideal triangles coming from $(g-3)$ basic blocks inside $a$, and $2(g-3)-1$ ideal triangles coming from adding in bands $a_i$, so a total of $(4g-13)$ ideal triangles. The total number of ideal triangles inside bands $b$ and $c$ is 3 each. So the total number of ideal triangles is $(4g-13)+2(3)+(3) = 4g-4$.

\subsubsection{Case 2: $m=1, g \geqslant 3$} In outer block 1, we let the single puncture sit inside band $c$. Similar to Figure~\ref{stuffing}, in a separate picture, consider $(g-1)$ basic blocks laid side to side, along the top of $I$. Superimpose this figure over outer block 1 such that the last basic block on the right is inserted inside band $b$ and the rest of the basic blocks are inserted inside band $a$. The region inside band $a$ and outside the bands inserted inside $a$, is an ideal polygon with $2(g-2)+1$ cusps. Divide this ideal polygon into $2(g-2)-1$ ideal triangles by adding bands $a_i$ for $i = 1, \cdots, 2(g-2)-2$, as in Figure~\ref{add-bands}.

The resulting interval exchange is a non-classical interval exchange with a Dehn twist sequence, on a surface with genus $g$ with $g \geqslant 3$ and a single puncture. It is easy to check that except the inside of band $c$ which is a once-punctured monogon, all complementary regions are ideal triangles. The Dehn twist sequence is similar to the previous case. A counting argument similar to the previous case shows that the complementary region consists of $(4g-3)$ ideal triangles and a single once punctured monogon, and hence we are done.

\subsubsection{Case 3: $m =2, g \geqslant 2$} In outer block 1, let the two punctures sit inside bands $b$ and $c$. Similar to Figure~\ref{stuffing}, consider $(g-1)$ basic blocks laid side to side, along the top of $I$. Superimpose this figure over outer block 1 such that all the basic blocks are inserted inside band $a$. The region inside band $a$ and outside the basic blocks inserted inside $a$, is an ideal polygon with $2(g-1)+1$ cusps. Divide this ideal polygon into $2(g-1)-1$ ideal triangles by adding bands $a_i$ for $i = 1, \cdots 2(g-1)-2$, as in Figure~\ref{add-bands}.

The counting argument shows that the complementary regions consist of $(4g-2)$ ideal triangles and two once-punctured monogons. So the resulting interval exchange is a non-classical interval exchange with a Dehn twist sequence on a surface of genus $g$ with $g \geqslant 2$ and two punctures.

\subsubsection{Case 4: $m \geqslant 3, g \geqslant 1$} In outer block 1, let two of the punctures sit inside bands $b$ and $c$. Similar to Figure~\ref{stuffing}, consider $(g-1)$ basic blocks laid side to side, along the top of $I$. In addition, construct $(m-2)$ orientation reversing bands laid side to side to the right of these basic blocks. Superimpose this figure over outer block 1 such that all the bands in the picture are inserted inside band $a$. The region inside band $a$ and outside the bands inserted inside $a$, is an ideal polygon with $2(g-1)+(m-2)+1$ cusps. Divide this ideal polygon into $2(g-2)+(m-2)-1$ ideal triangles by adding bands $a_i$ for $i = 1, \cdots, 2(g-2)+(m-2)-2$, as in Figure~\ref{add-bands}.

The counting argument shows that the complementary regions consist of $(4g-4+m)$ ideal triangles and $m$ once punctured monogons. So the resulting interval exchange is a non-classical interval exchange with a Dehn twist sequence on a surface of genus $g$ with $g \geqslant 1$ and $m \geqslant 3$ punctures.

\subsubsection{Case 5: $m \geqslant 5, g=0$} In outer block 2, two of the punctures are already accounted for. Let two other punctures sit inside bands $b$ and $c$. Consider $(m-4)$ orientation reversing bands laid side to side, along the top of $I$. Superimpose this figure over outer block 2 such that all the bands are inserted inside band $a$. The region inside band $a$ and outside the inserted bands is an ideal polygon with $(m-4)+1$ cusps. Divide this ideal polygon into $(m-4)-1$ ideal triangles by adding bands $a_i$ for $i = 1, \cdots, (m-4)-2$, as in Figure~\ref{add-bands}.

The counting argument shows that the complementary regions consist of $(m-4)$ ideal triangles and $m$ once punctured monogons. So the resulting interval exchange is a non-classical interval exchange with a Dehn twist sequence on a sphere with $m$ punctures where $m \geqslant 5$.

\subsubsection{The remaining low complexity surfaces:} The cases above exclude some low complexity examples. In each of these, we directly draw the non-classical interval exchange, and leave the details to the reader.

\begin{enumerate}
\item $\mathit{m=0, g=2:}$ Genus 2 in Figure~\ref{genus2}.
\item $\mathit{m=0, g=3:}$ Genus 3 in Figure~\ref{genus3}. Here the two instances of the same alphabet represent the two ends of the same band.
\item $\mathit{m=1, g=2:}$ Once-puncture genus 2 in Figure~\ref{genus2punc1}.
\item $\mathit{m=1, g=1:}$ In the once punctured torus case, the interval exchange is the same as the classical interval exchange with 2 bands that we considered for the torus.
\item $\mathit{m=2, g=1:}$ Twice punctured torus in Figure~\ref{genus1punc2}.
\item $\mathit{m=4, g=0:}$ 4-times punctured sphere in Figure~\ref{4puncsphere}.
\end{enumerate}
\begin{figure}[htb]
\begin{center}
\ \psfig{file=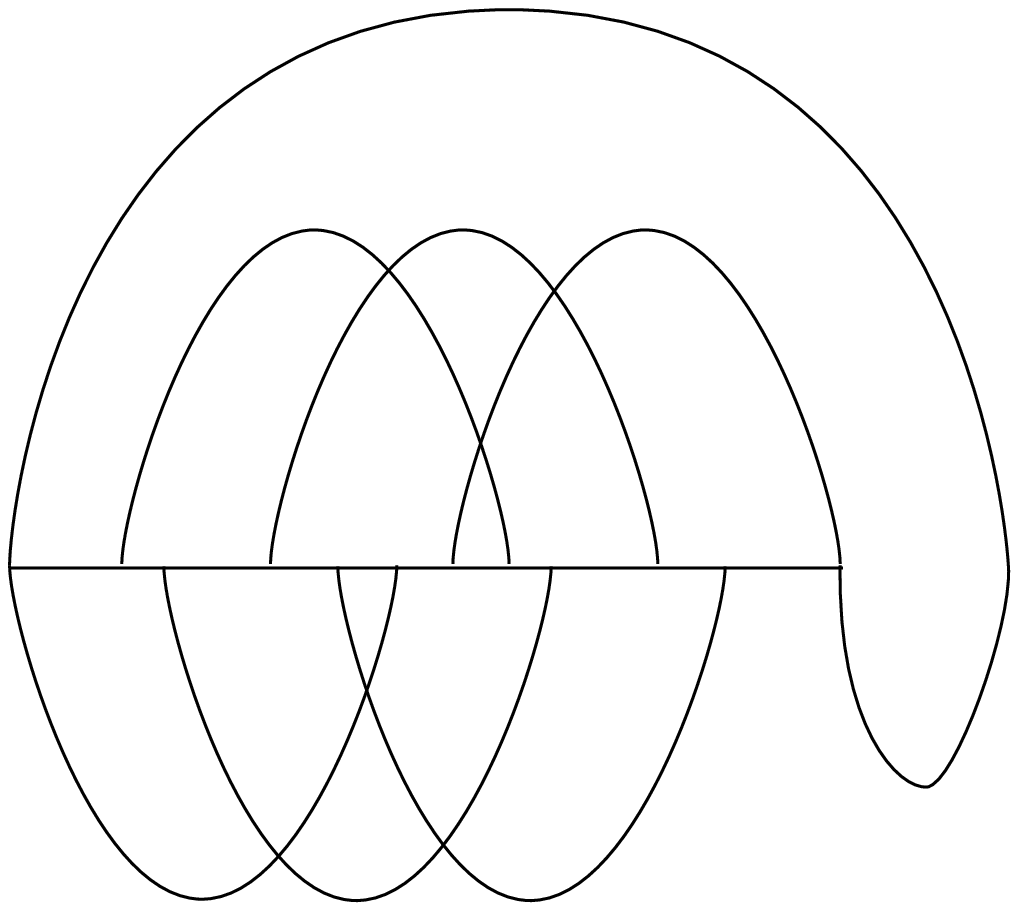, height=1truein, width=4truein} \caption{Genus 2} \label{genus2}
\end{center}
\setlength{\unitlength}{1in}
\begin{picture}(0,0)(0,0)
\end{picture}
\end{figure}
\begin{figure}[htb]
\begin{center}
\ \psfig{file=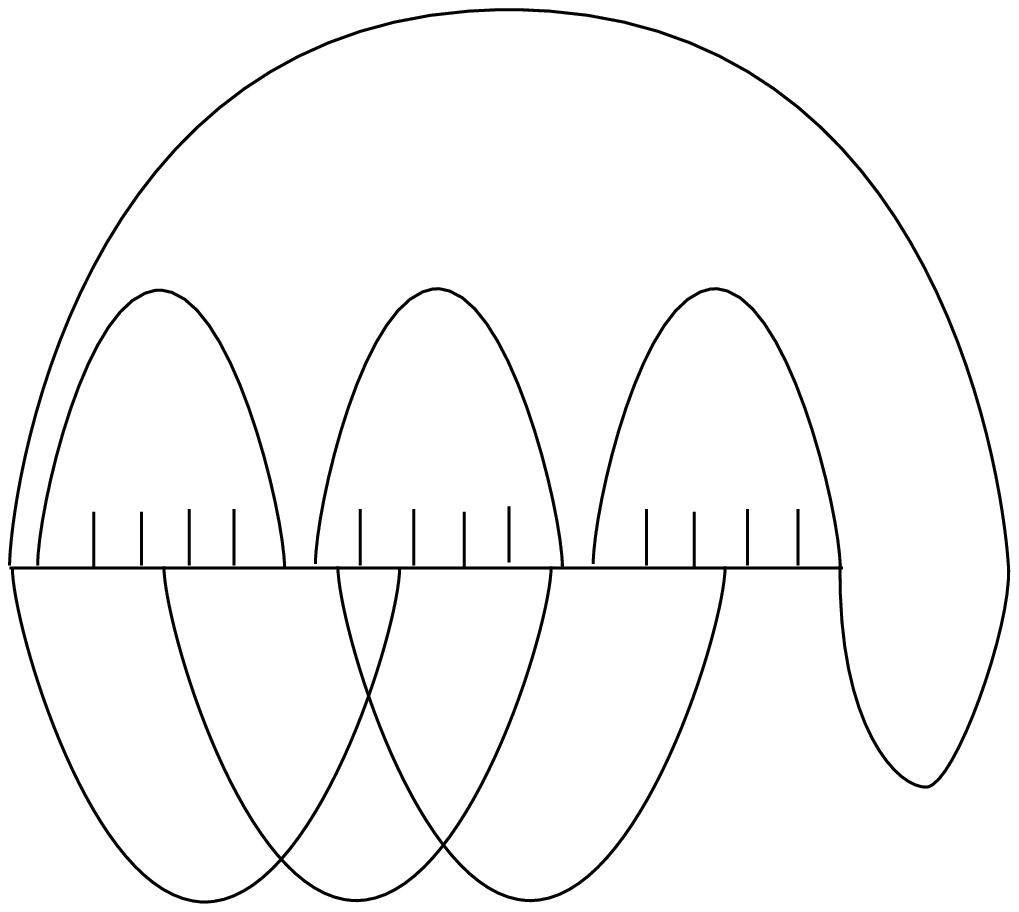, height=1.1truein, width=4truein} \caption{Genus 3} \label{genus3}
\end{center}
\setlength{\unitlength}{1in}
\begin{picture}(0,0)(0,0)
\put(-1.65,1.1){$a$}\put(-0.4,1.1){$a$}\put(-1.1,1.1){$b$} \put(-0.2,1.1){$b$} \put(-0.6,1.1){$c$} \put(0.75,1.1){$c$} \put(0,1.1){$d$} \put(0.95,1.1){$d$} \put(0.55,1.1){$e$} \put(-1.45,1.1){$e$}\put(1.1,1.1){$f$}\put(-1.3,1.1){$f$}
\end{picture}
\end{figure}
\begin{figure}[htb]
\begin{center}
\ \psfig{file=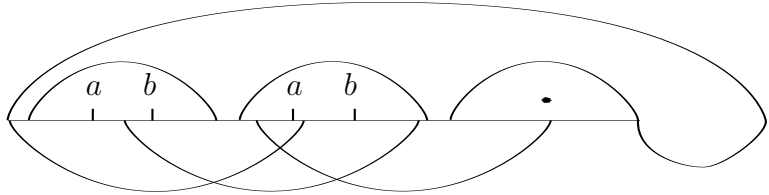, height=1truein, width=4truein} \caption{Once-punctured Genus 2} \label{genus2punc1}
\end{center}
\setlength{\unitlength}{1in}
\begin{picture}(0,0)(0,0)
\put(-1.55,1.1){$a$}\put(-0.5,1.1){$a$}\put(-1.25,1.1){$b$}  \put(-0.2,1.1){$b$}
\end{picture}
\end{figure}
\begin{figure}[htb]
\begin{center}
\ \psfig{file=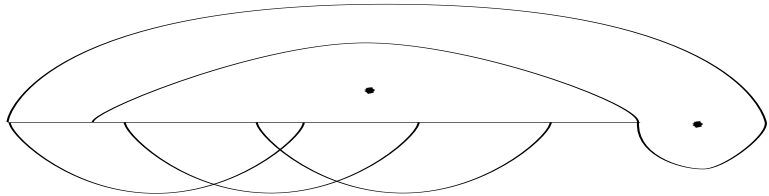, height=1truein, width=4truein} \caption{Twice-punctured Genus 1} \label{genus1punc2}
\end{center}
\setlength{\unitlength}{1in}
\begin{picture}(0,0)(0,0)
\end{picture}
\end{figure}
\begin{figure}[htb]
\begin{center}
\ \psfig{file=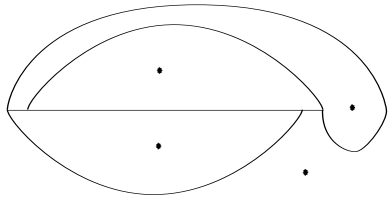, height=1truein, width=2truein} \caption{4 punctured sphere} \label{4puncsphere}
\end{center}
\setlength{\unitlength}{1in}
\begin{picture}(0,0)(0,0)
\end{picture}
\end{figure}

\subsection{Geometric Irreducibility:} From Definition 3.1 of Boissy and Lanneau \cite{Boi-Lan}, it is easy to check directly that the combinatorics of the initial non-classical interval exchange that we constructed in each case is geometrically irreducible. So by the results of \cite{Boi-Lan}, the combinatorics lies in some attractor of the Rauzy diagram. Since the attractor is connected in the directed sense, the combinatorics at each stage of the expansion, lies in the attractor. Henceforth, instead of the full Rauzy diagram, we focus just on the attractor containing our initial combinatorics, and denote it by $\overline{\cG}$.

\subsection{Terminology:} We shall denote the combinatorics constructed by $\phi_0$, and the configuration space at $\phi_0$ by $W_0$. In all cases, the spine of the union of the base interval $I$ and the band $B$ is the vertex cycle about which we have a Dehn twist sequence. We shall call the vertex cycle the {\em stable} vertex cycle and denote it by $v$. We shall denote the Dehn twist sequence by $\jmath_0$.

\subsection{Remarks about the combinatorics $\phi_0$:} \label{wind} Under any embedding of $\phi_0$ into the surface, the stable vertex cycle $v$ gives a separating curve on the surface. Secondly, suppose we fix an orientation on $v$ by running along $B$ from the bottom end to the top end and back along $I$. Let $\gamma$ be any simple closed curve carried by the interval exchange. For every intersection point $p$ of $\gamma$ with the base interval $I$, as one follows $\gamma$ from the top to the bottom through $p$, it can wind around $v$ only in the positive direction. This property gives important consequences for the sub-surface projection to the annulus with core curve $v$, as we shall see later.

\subsection{The Lebesgue measure estimate for iterations of the Dehn twist sequence:} Thought of as a directed path in $\overline{\cG}$, the Dehn twist sequence $\jmath_0$ returns us to the same vertex $\phi_0$. The matrix $Q_0$ associated to $\jmath_0$ has the effect that it adds the column $e_B$ twice to all columns $e_\alpha$ associated to bands $\alpha$ on top. Let $n \jmath_0$ denote the splitting sequence $\jmath_0 \ast \jmath_0 \ast \cdots \ast \jmath_0$ i.e. the sequence given by iterating $\jmath_0$, $n$ successive times. Let $Q_n$ be the matrix associated to it i.e. $Q_n = (Q_0)^n$. We have the estimate

\begin{proposition}\label{Dehn}
There exists a positive integer $j$ such that $\ell(JQ_n(W_0)) \approx 1/n^j$ i.e. there is a constant $a_0>1$ such that
\begin{equation}\label{vol}
\frac{1}{a_0 n^j} < \ell(JQ_n(W_0)) < \frac{a_0}{n^j}
\end{equation}
for all $n$.
\end{proposition}

\begin{proof}
The band $B$ is the only orientation preserving band. This means that $W_0$ is a cone to $e_B$ of a convex subset of the face $F_B$ of $\Delta$ opposite $e_B$. For every pair of orientation reversing bands $\alpha$ on top and $\beta$ on bottom there is a vertex $e_{\alpha \beta}$ of $W_0$, where $e_{\alpha \beta}$ is the midpoint of the edge joining the vertices $e_\alpha$ and $e_\beta$ of the standard simplex $\Delta$. Thus $W_0$ is cone to $e_B$ of the convex hull of the vertices $e_{\alpha \beta}$. A schematic picture of $W_0$ is shown in Figure~\ref{dehntwistestimate}.
\begin{figure}[htb]
\begin{center}
\ \psfig{file=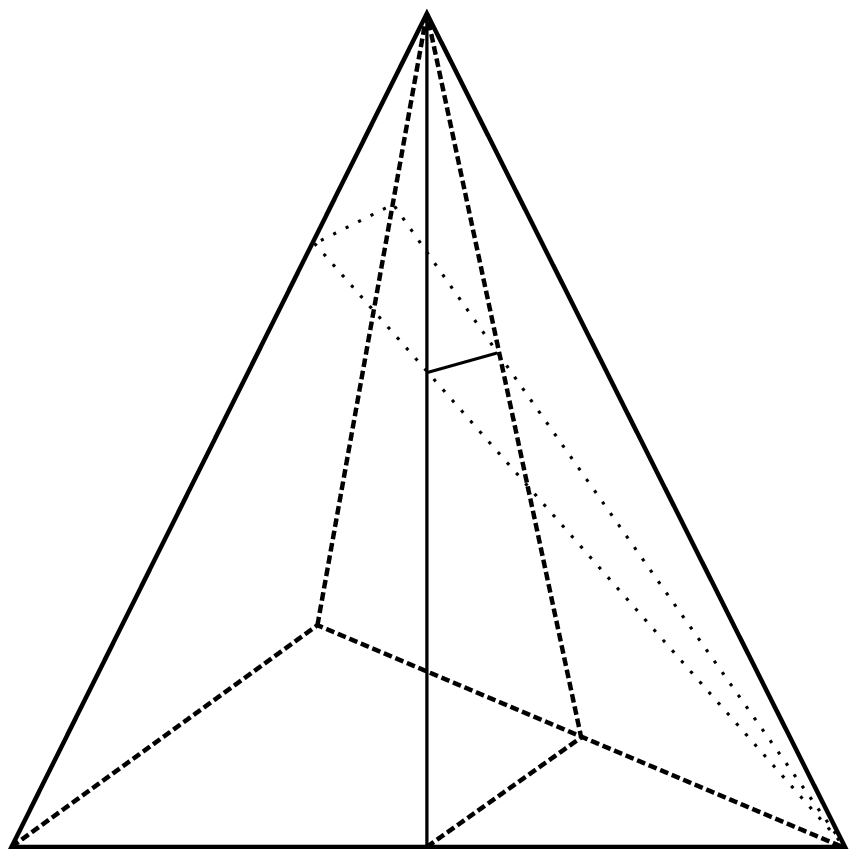, height=2truein, width=2.5truein} \caption{} \label{dehntwistestimate}
\end{center}
\setlength{\unitlength}{1in}
\begin{picture}(0,0)(0,0)
\put(0,2.7){$e_B$}\put(0,0.5){$e_{\alpha \beta}$}\put(0.5,1){$e_{\alpha \gamma}$}\put(-0.2,1.6){$f_{\alpha \beta}$}\put(0.3,1.8){$f_{\alpha \gamma}$}
\end{picture}
\end{figure}

The subset $JQ_n(W_0)$ of $W_0$ is a convex set. Hence, our goal is to identify the vertices of $JQ_n(W_0)$. For $\alpha$ on top, the column $Q_n(\alpha)= e_\alpha + (2n) e_B$. All other columns of $Q_n$ are the same as the corresponding columns of the identity matrix. This means, first, that $e_B$ is a vertex of $JQ_n(W_0)$ and second, as shown in Figure~\ref{dehntwistestimate}, on every side of $W_0$ that joins $e_B$ to some $e_{\alpha \beta}$, there is a vertex $f_{\alpha \beta}$ of $JQ_n(W_0)$, whose linear combination is
\[
f_{\alpha \beta} = \frac{n}{n+1} e_B + \frac{1}{n+1}e_{\alpha \beta}
\]
The subset $JQ_n(W_0)$ is the cone to $e_B$ of the convex hull of the vertices $f_{\alpha \beta}$. From the linear combination it is clear that for $j= d-2$, where $d$ is the total number of bands, there is a constant $a_0>1$ such that the $(d-2)$-volume of $JQ_n(W_0)$ is related to the $(d-2)$-volume of $W_0$ by Estimate~\eqref{vol}.
\end{proof}

\section{The complex of curves}\label{Curvecomplex}

With the exception of the torus, once-punctured torus and the 4-punctured sphere, the curve complex $\cC(\Sigma)$ of a surface $\Sigma$ is a locally infinite simplicial complex whose vertices are the isotopy classes of essential, non-peripheral, simple closed curves on $\Sigma$. A collection of vertices span a simplex if there are representatives of the curves that can be realized disjointly on the surface. In the low complexity examples that are the exceptions the definition is modified: two vertices are connected by an edge if there are representatives of the simple closed curves intersecting minimally. The mapping class group acts on the curve complex in the obvious way. For detailed background on the curve complex, see the influential paper \cite{Mas-Min1} by Masur and Minsky. Alternatively, see Bowditch \cite{Bow}.

Of primary interest is the coarse geometry of $\cC(\Sigma)$. The curve complex $\cC(\Sigma)$ is quasi-isometric to its 1-skeleton with the path metric on it. The 1-skeleton is a locally infinite graph with infinite diameter. The random walk on the mapping class group can be projected to this graph by using the group action.

Masur and Minsky \cite{Mas-Min1} showed that the curve complex is $\delta$-hyperbolic. For $\delta$-hyperbolic spaces, it is possible to construct a natural boundary at infinity, called the Gromov boundary. Roughly speaking, points on the Gromov boundary correspond to equivalence classes of infinite geodesic rays under a suitable equivalence relation. It was shown by Klarreich \cite{Kla} that the Gromov boundary $\partial \cC(\Sigma)$ is the space $\fmin$ of {\em minimal} foliations on the surface i.e. foliations on $\Sigma$ that have a non-zero intersection number with every simple closed curve. See also \cite{Ham}.

There is a coarse distance non-increasing map $q$ from the \teichmuller space $T(\Sigma)$ to the curve complex $\cC(\Sigma)$ defined as follows: A point in \teichmuller space gives a marked hyperbolic structure on $\Sigma$. The map $q$ is defined by sending the point to the shortest curve or the {\em systole} in this hyperbolic metric. If there is more than one shortest curve, then it turns out that the set of shortest curves has bounded diameter in $\cC(\Sigma)$, where the bound depends only on the topology of $\Sigma$. So picking one of the curves defines $q$ coarsely. It is clear from the definition that $q$ is coarsely equivariant with respect to the mapping class group action.

Here is another way to set up the definition of $q$: For a constant $\epsilon > 0$ and a simple closed curve $\alpha$, define the $\epsilon$-thin part of $T(\Sigma)$ corresponding to $\alpha$, to be those points in $T(\Sigma)$ for which the geodesic representative of $\alpha$ has length less than $\epsilon$. If $\epsilon$ is smaller than a universal constant, the intersection pattern of the thin parts corresponding to different simplex closed curves, is modeled by the curve complex. The thin parts have unbounded diameter in the \teichmuller metric, but one can form the {\em electrified} \teichmuller space $T_{el}(\Sigma)$ by adding an extra point for each thin part, and connecting every point in that thin part to it's new point by an edge of length 1. This has the effect of collapsing the thin parts to bounded diameter. There is an obvious inclusion of $T(\Sigma)$ into $T_{el}(\Sigma)$. The electrified \teichmuller space $T_{el}(\Sigma)$ is quasi-isometric to the curve complex, giving us another definition of the map $q$. See \cite{Mas-Min1}, Lemma 3.1.

The precise statement of Klarreich's theorem \cite{Kla} is

\begin{theorem}\label{Klarreich}
The inclusion map $q: T(\Sigma) \to T_{el}(\Sigma)$ extends continuously to the subset $\pmf_{min}$ of projective classes of measured minimal foliations, to give a map $\partial q: \pmf_{min} \to \partial T_{el}(\Sigma)$. The map $\partial q$ is surjective and $\partial q(F) = \partial q(G)$ if and only if $F$ and $G$ are topologically equivalent minimal foliations. Moreover, the image under $q$ of any sequence $X_n$ in $T(\Sigma)$ that converges to a point in $\pmf \setminus \pmf_{min}$ cannot accumulate at any point of $\partial T_{el}(\Sigma)$.
\end{theorem}

In fact, image by $q$ of a \teichmuller geodesic is an un-parameterized quasi-geodesic in $\cC(\Sigma)$. The map $\partial q$ is a quotient map in the sense that it takes a measured minimal foliation and forgets the measure. This quotient map is a bijection when restricted to the set $\ue$ of the uniquely ergodic foliations.

An immediate consequence of Klarreich's theorem and Theorem~\ref{Poisson} of Kaimanovich and Masur is the theorem (See \cite{Mah1}, Theorem 5.1)

\begin{theorem}\label{poisson}
Let $\mu$ be an initial distribution on the mapping class group $G$, such that the subgroup generated by the support of $\mu$ is non-elementary. Then, for any base point $\bb$ in $\cC(\Sigma)$ and almost every sample path $\p$, the sequence $\omega_n \bb$ converges to a point in $\partial \cC(\Sigma) = \fmin$. Let $\nu$ denote the induced harmonic measure on $\fmin$. Then, the measure $\nu$ is the push-forward under $\partial q$ of the induced harmonic measure on $\pmf$.
\end{theorem}

\begin{proof}
By Theorem~\ref{Poisson} of Kaimanovich and Masur, for any base-point $X$ in $T(\Sigma)$ and for almost every sample path $\p= \{\omega_n\}$, the sequence $\omega_n X$ in $T(\Sigma)$ converges to a uniquely ergodic foliation in $\pmf$. Uniquely ergodic foliations are minimal, so by Klarreich's theorem the image in $T_{el}(\Sigma)$ of the sequence $\omega_n X$ converges to the same foliation in $\partial \cC(\Sigma)$. Here, we have used the quasi-isometry between $T_{el}(\Sigma)$ and $\cC(\Sigma)$ and chosen $X$ such that the image of $X$ in $\cC(\Sigma)$ is $\bb$. The random walk thus projected defines a harmonic measure on $\partial \cC(\Sigma)$. The quotient map $\partial q: \fmin \to \partial \cC(\Sigma)$ is injective on $\ue$. So the harmonic measure on $\cC(\Sigma)$ is a push-forward of the harmonic measure $\nu$ on $\pmf$, and we shall continue to call it $\nu$.
\end{proof}

\section{Subsurface projections}\label{Projections}

For details about subsurface projections, see \cite{Mas-Min2} or \cite{Bow}. Let $Y$ be an essential connected subsurface of $\Sigma$, not a thrice punctured sphere or an annulus. Define a subsurface projection map $\pi$ from $\cC(\Sigma)$ to $\cC(Y)$ as follows: intersect a simple closed curve $x$ with the subsurface $Y$, and let $N(x,Y)$ be a regular neighborhood in $Y$ of the union of $x \cap Y$ and $\partial Y$. Set $\pi(x)$ to be one of the components of $\partial N(x,Y)$, provided the component chosen is not parallel to a component of $\partial Y$. Strictly speaking, this is defined only when the curve intersects $Y$ essentially, but the set of curves $S_Y$ that do not, has diameter 2 in $\cC(\Sigma)$. There is also a choice involved in selecting the component. However, the different choices result in a set of vertices, whose diameter in $\cC(Y)$ is bounded above by a constant that depends only on the topology of $\Sigma$. This gives a coarse definition of the map $\pi$. Since $\pi$ fails to be defined only on a bounded set $S_Y$, it extends to the boundary $\partial \cC(\Sigma)$. The mapping class group of $Y$ is a subgroup of the mapping class group of $\Sigma$ in a natural way, and the map $\pi$ is equivariant with respect to its action.

It requires more care to set up the definition of the subsurface projection to an essential annulus $A$ with core curve $x$. The goal is to define a complex $\cC(x)$ associated to $A$, and then define the subsurface projection map $\pi: \cC(\Sigma) \to \cC(x)$ such that $\pi$ records the twisting of a curve around $x$. For this, one would like $\cC(x)$ to be simply $\Z$, but there is no natural way to do this. Instead, let $\widetilde{\Sigma}$ be the annulus cover of $\Sigma$ corresponding to $A$. This cover is a quotient of $\H^2$ in a natural way. So one obtains a natural compactification $\widehat{\Sigma}$ of $\widetilde{\Sigma}$ from the compactification of $\H^2$ to the closed disk. Define $\cC(x)$ as follows: Let the vertices of $\cC(x)$ be paths connecting the two boundary components of $\widehat{\Sigma}$ modulo homotopies that fix the endpoints. Two vertices are connected by an edge if there are representatives of the vertices with no intersection in $\widetilde{\Sigma}$. Fixing an orientation of $\Sigma$ and ordering the components of $\partial \widehat{\Sigma}$, we can define an algebraic intersection number of two vertices $u$ and $v$, denoted by $u \cdot v$. For distinct vertices $u$ and $v$ of $\cC(x)$, the distance in $\cC(x)$ between $u$ and $v$ can be shown to be $1+ \vert u \cdot v \vert$. Moreover, after fixing a vertex $u \in \cC(\Sigma)$, it can be checked that the map $v \to v \cdot u$ gives a quasi-isometry between $\cC(x)$ and $\Z$. In addition, the quasi-isometry constants are independent of the choice $u$. To define the subsurface projection $\pi(y)$ of a curve $y$ intersecting $A$ essentially, consider a lift $\widehat{y}$ in $\widehat{\Sigma}$ and choose a component of this lift running from one boundary component of $\widehat{\Sigma}$ to the other. The set of various components is of finite diameter in $\cC(x)$ and so the subsurface projection map $\pi$ to $\cC(x)$ is coarsely well-defined. Finally the map $\pi$ has the property that if $D_x$ denotes the Dehn twist about $x$, then
\begin{equation}\label{equivariance}
d_{\cC(x)}(\pi(D_x^n(y)),\pi(y)) = 2 + |n|
\end{equation}
Thus, defining $\pi$ this way achieves the desired property of recording the twisting around $x$. There is a natural $\Z$ action on $\cC(x)$ by Dehn twisting around the core curve of the annular cover $\widehat{\Sigma}$. The group $\Z$ also has an inclusion into the mapping class group of $\Sigma$ as Dehn twists around $x$, and so it acts on $\cC(\Sigma)$ through this inclusion. The subsurface projection map $\pi$ is coarsely equivariant with respect to the $\Z$ action on $\cC(\Sigma)$ and $\cC(x)$.

\subsection{The relative space:}
For technical reasons related to the proof of Theorem~\ref{exp-decay}, we want to define sub-surface projections for the mapping class group itself. For this purpose, we introduce two spaces: the {\em marking complex} and the {\em relative space}. The marking complex is quasi-isometric to the mapping class group itself. The relative space is obtained by electrifying the mapping class group or alternatively, the marking complex because of the quasi-isometry between the two, and is quasi-isometric to the curve complex.

Suppose $G$ is a group with a finite symmetric generating set $A$. For each group element $g$, one defines the word length of $g$ with respect to $A$ as the length of the shortest word in the generators representing $g$. This length function defines a left invariant metric on $G$ which we call the {\em word metric}. The word metric can be recovered from a graph associated to $(G,A)$ called the {\em Cayley graph}. The vertices of the Cayley graph are the group elements, and two elements $g$ and $h$ are connected by an edge of length 1 if $g^{-1}h$ is a generator. The word metric on $G$ is simply the path metric on the Cayley graph. It is easy to check that different choices of finite symmetric generating sets result in Cayley graphs that are quasi-isometric to each other.

Given a group $G$ with a set of generators $A$, and a collection of subgroups $\mathcal{H}=\{ H_i\}$, the {\em relative} length of a group element $g$ is defined to be the length of the shortest word representing $g$ in the (typically infinite) set of generators $\mathcal{H} \cup A$. This defines a metric on $G$ called the {\em relative metric}. We shall denote $G$ with this metric by $\widehat{G}$, and call this the {\em relative space}.

Now let $G$ be the mapping class group of $\Sigma$. It is a classical theorem that the mapping class group is finitely generated. We shall fix a favorite set of symmetric generators for $G$ once and for all. The metric on $G$ shall be implicitly assumed to be the word metric with respect to this chosen set.

We can consider the relative metric on $G$ with respect to the following collection of subgroups: There are finitely many orbits of simple non-peripheral closed curves in $\Sigma$ under the action of the mapping class group. Let $\{ a_1, \ldots, a_r\}$ be a list of representatives of these orbits, and let $H_i$ be the subgroup of the mapping class group that fixes $a_i$.

It was shown by Masur and Minsky \cite{Mas-Min1}, that the resulting relative space $\widehat{G}$ for the mapping class group is quasi-isometric to $\cC(\Sigma)$.

\subsection{The marking complex:} Let $\{ x_1, \ldots, x_n \}$ be a simplex in $\cC(\Sigma)$ i.e. the set $\{x_1 \ldots, x_n\}$ is a set of disjoint simple closed curves in $\Sigma$. A {\em marking} in $\Sigma$ is a set $m= \{ p_1, \ldots, p_n \}$, where either $p_i = x_i$, or $p_i$ is a pair $(x_i, t_i)$, where $t_i$ is a diameter-one set of vertices of the complex $\cC(x_i)$ associated to the annulus with core curve $x_i$. The $x_i$ are called the {\em base curves} of the marking and the $t_i$, when defined, are called the transversals. A marking $m$ is {\em complete} if the simplex formed by the base is a maximal simplex in $\Sigma$ and if every curve $x_i$ has a transversal. In other words, a complete marking consists of a pants decomposition of $\Sigma$ along with the choice of a transversal to each cuff in the pants decomposition.

Given $x \in \cC(\Sigma)$, a {\em clean transverse curve} for $x$ is a curve $y$ such that a regular neighborhood $F$ for $x \cup y$, is either a once punctured torus or a four times punctured sphere, and $x$ and $y$ are adjacent in the curve complex $\cC(F)$ of $F$. A marking is {\em clean} if every $p_i$ is a pair of the form $(x_i, \pi_{x_i}(y_i) )$, where $y_i$ is a clean transverse curve for $x_i$ disjoint from all other curves in $\mathit{base}(m)$.

Complete clean markings on $\Sigma$ can be made into an infinite graph $M(\Sigma)$, by adding edges corresponding to the following {\em elementary moves}: Consider a marking $m$ and a base curve $x$ of $m$. A marking $m'$ is obtained from $m$ by a {\em twist} move about $x$ if $\mathit{base}(m) = \mathit{base}(m')$, the transversals for all base curves except $x$ are the same, and if $t(x)$ and $t'(x)$ are the transversals to $x$ in the markings $m$ and $m'$ respectively, then $d_{\cC(x)}(t(x),t'(x)) \leqslant  2$. In other words, the transversal to $x$ in $m'$ has a projection to $\cC(x)$ that differs from the projection of the transversal to $x$ in $m$, by at most two Dehn twists about $x$. A marking $m'$ is obtained from a marking $m$ by a {\em flip move} along $x$ if there exists $x'$ in $\mathit{base}(m')$ such that $\mathit{base}(m') \setminus \{x\} = \mathit{base}(m) \setminus \{x'\}$, a regular neighborhood $F$ of $x \cup x'$ is either a 1-holed torus or a 4-holed sphere, and $d_{\cC(F)}(x,x') = 1$, and for the respective transversals $t(x)$ and $t(x')$ to $x$ and $x'$ in $m$ and $m'$, we have $d_{\cC(x)}(t(x), \pi_x(x')) \leqslant  2$ and $d_{\cC(x')}(t(x'), \pi_{x'}(x)) \leqslant  2$, where $\pi_x$ and $\pi_{x'}$ are the subsurface projections to the respective annuli.

In \cite{Mas-Min2}, Masur and Minsky showed that the space $M(S)$ of complete clean markings with the path metric defined above is quasi-isometric to the mapping class group $G$ with the word metric. Let $R_1$ be the Lipshitz constant for this quasi-isometry i.e. the multiplicative constant in the quasi-isometry.

\subsection{Subsurface projections on the marking complex:}\label{Proj}

For a subsurface $Y$, we will write $d_Y$ for distance in the complex of curves $\cC(Y)$ of $Y$.

The subsurface projection map $\pi$ on the curve complex can be extended to a subsurface projection map for complete clean markings. For a complete clean marking $m$, define $\pi(m)$ as follows: if $Y$ is an annulus with core curve $x \in \mathit{base}(m)$, then define $\pi(m) = \pi(t)$, where $t$ is the transversal for $x$. Otherwise set $\pi(m) = \pi(\mathit{base}(m))$. As before, this gives a coarse definition for $\pi$. The map $\pi$ is well defined on the entire space of complete clean markings $M(\Sigma)$, unlike the sub-surface projection map on $\cC(\Sigma)$.

If distinct markings $m$ and $m'$ differ by a twist move about $x$, then $\pi(m) = \pi(m')$ unless $Y$ is the annulus with core curve $x$, in which case $d_Y(\pi(m), \pi(m')) = d_Y(\pi(t), \pi(t')) \leqslant  2$. If the markings $m$ and $m'$ differ by a flip move along $x$, then in a similar way one can check that $d_Y(\pi(m), \pi(m')) \leqslant  2$. This implies that for any sub-surface $Y$, the projection map $\pi: M(\Sigma) \to \cC(Y)$ is coarsely 2-Lipshitz. By pre-composing $\pi$ on $M(\Sigma)$ by the quasi-isometry from $G$ to $M(\Sigma)$ we get a projection $\pi$ on $G$, which we continue to denote by $\pi$. The map $\pi$ on $G$ is then coarsely $2R_1$-Lipshitz.

In \cite{Mas-Min2}, Masur and Minsky proved a quasi-distance formula expressing the distance in the marking complex $M(\Sigma)$ in terms of subsurface projections. Here, we state a slightly simplified version of it. Given a number $A>0$, for any $d \in \N$, let
\begin{eqnarray*}
[d]_A &=& d \hskip 20pt \text{if} \hskip 5pt d \geq A \\
&=& 0 \hskip 20pt \text{otherwise}
\end{eqnarray*}

\begin{theorem}[Quasi-distance formula]\label{Quasi-Distance}
There exists a constant $A>0$ that depends only on the topology of $\Sigma$, such that for any pair of markings $m$ and $m'$ in $M(\Sigma)$ we have the estimate
\begin{equation}\label{quasi-distance}
d_{M(\Sigma)}(m,m') \approx \sum_{Y \subseteq \Sigma} [d_Y(\pi_Y(m),\pi_Y(m'))]_A
\end{equation}
where the sum runs over all sub-surfaces $Y$ of $\Sigma$. The constants of approximation in the above formula depend only on the topology of $\Sigma$.
\end{theorem}

The map $\pi$ can also be thought of as a subsurface projection map on the relative space $\widehat{G}$. Recall that the relative space is quasi-isometric to $\cC(\Sigma)$. If we pull-back the subsurface projection on $\cC(\Sigma)$ by this quasi-isometry, we get a map $\widehat{G} \to \cC(Y)$ that is coarsely equivalent to the map $\pi$ defined above. The reason for defining projection $\pi$ on $\widehat{G}$ this way is that now $\pi$ is defined everywhere on $\widehat{G}$. This feature is exploited in the proof of Theorem~\ref{exp-decay}.

We will write $\widehat{d}$ for distance in the relative metric on $G$.

\section{Useful facts about the relative space}\label{Facts}

In this section, we state some facts due to Maher, about half-spaces in the relative space. These shall be used in the proof of Theorem~\ref{exp-decay} in the next section. Some results stated here are straightforward, and the proofs are left to the reader.

Since the relative space $\widehat{G}$ is $\delta$-hyperbolic, nearest point projections are coarsely well defined. Denote the identity element in $G$ by 1. Fixing 1 as the base point, we can define the Gromov product $(x \mid y)$ to be
\[
(x \mid y) = \frac{1}{2} \left( \widehat{d}(1,x)+ \widehat{d}(1,y)- \widehat{d}(x,y) \right)
\]
It turns out that the points in the Gromov boundary $\fmin$ correspond to equivalence classes of sequences $\x = (x_i)$, where $\x \sim \y$ if and only if the Gromov product $(x_i \mid y_j) \to \infty$ as $i,j \to \infty$. It can be checked that the equivalence relation does not depend on the base point, and so this can be taken as a definition of the Gromov boundary.

Now we state the results about half-spaces.

\begin{proposition}\label{npp}
Let $p$ and $q$ be nearest points to $z$ on a geodesic $[x,y]$. Then $\widehat{d}(p,q) \leqslant  6\delta$.
\end{proposition}

\begin{proposition}[Maher \cite{Mah2} Proposition 3.4] \label{double}
Let $[x,y]$ be a geodesic and let $p$ be a closest point on $[x,y]$ to $a$, and let $q$ be a closest point on $[x,y]$ to $b$. If $\widehat{d}(p,q) > 14\delta$ then $\widehat{d}(a,b) \geqslant \widehat{d}(a,p) + \widehat{d}(p,q) +
\widehat{d}(q,b) - 24\delta$.
\end{proposition}

\noindent In particular, the above proposition implies that nearest point projection is coarsely distance non-increasing.

\begin{corollary}\label{cdd}
Let $[x, y]$ be a geodesic, and let $p$ be a nearest point to $a$ and $q$ a nearest point to $b$ on $[x,y]$. Then
$\widehat{d}(p,q) \leqslant  \widehat{d}(a,b) + 24 \delta$.
\end{corollary}

\noindent Let $H(x, y)$ be the half-space of points in $\widehat{G}$ that are closer to $y$ than to $x$ i.e.
\[
H(x, y)  = \{ a \in \widehat{G} \mid \widehat{d}(y, a) \leqslant  \widehat{d}(x, a) \}
\]

\begin{proposition} [Maher \cite{Mah2} Proposition 3.7] \label{half}
Let $z \in H(x,y)$, and let $p$ be the nearest point to $z$ on a geodesic $[x,y]$. Then $\widehat{d}(y,p) \leqslant  (1/2)\widehat{d}(x,y) + 3\delta$. Conversely, if $\widehat{d}(y,p) \leqslant  (1/2)\widehat{d}(x,y) -3\delta$, then $z \in H(x,y)$.
\end{proposition}

\noindent Henceforth, we consider half-spaces based at 1 i.e. half-spaces of the form $H(1,a)$.

\begin{proposition}\label{boundedset}
There is a constant $K_1$, which only depends on $\delta$, such that for any half-space $H(1, a)$, with $\widehat{d}(1,a) \geqslant K_1$, and for any set $X$ of relative diameter at most $D$ intersecting $H(1,a)$, the
half-space $H(1,b)$ contains $H(1, a) \cup X$, where $b \in [1,a]$ with $\widehat{d}(1,b) = \widehat{d}(1,a) - 2D - K_1$.
\end{proposition}

\begin{proof}
By Proposition~\ref{half}, the nearest point projection of $H(1,a)$ to $[1, a]$ is distance at least $ (1/2)\widehat{d}(1,a) - 3\delta$ from $1$. By Corollary~\ref{cdd}, the nearest point projection of $H(1,a) \cup X$ to $[1,a]$ is then relative distance at least $(1/2)\widehat{d}(1,a) - 27\delta - D$ from $1$. So, if a point $b$ on $[1,a]$ satisfies
\begin{equation}\label{E1}
\frac{\widehat{d}(1,b)}{2} + 3\delta \leqslant \frac{\widehat{d}(1,a)}{2} - 27 \delta -D
\end{equation}
then for any point in $H(1,a) \cup X$, the closest point $p$ on $[1,b]$ satisfies $\widehat{d}(b,p) \leqslant (1/2)\widehat{d}(1,b) - 3 \delta$. By the last line in Proposition~\ref{half}, the set $H(1,a) \cup X$ has to lie in $H(1,b)$. Rewriting Inequality~\eqref{E1}, we see that we may choose $K_1 = 30 \delta$.
\end{proof}

\begin{proposition} \label{nbhd}
There is a constant $K_2$, which depends only on $\delta$, such that for any half-space $H(1,a)$, with $\widehat{d}(1,a) \geqslant K_2$, and for any large enough positive integer $r$, there is a point $b$ with $\widehat{d}(1,b) = r - K_2$ such that every half-space $H(1,x)$, with $\widehat{d}(1,x) = r$, that intersects $H(1,a)$, is contained in $H(1,b)$.
\end{proposition}

\begin{proof}
This follows immediately from the proof of Lemma 2.15 in \cite{Mah3}.
\end{proof}

\noindent Let $\overline{H(a, b)}$ be the limit set of the half space $H(a,b)$ in $\fmin= \partial \cC(\Sigma)$.

\begin{proposition} \label{zero}
Let $\mu$ be a probability distribution on $G$, whose support generates a non-elementary subgroup, and let $\nu$ be the corresponding harmonic measure. Then $\nu(\overline{H(1, x)}) \to 0$ as $\widehat{d}(1,x) \to \infty$.
\end{proposition}

\begin{proof}
Suppose not. Then for some $\epsilon > 0$, there is a sequence $x_i$ with $\widehat{d}(1,x_i) \to \infty$ such that $\nu(\overline{H(1,x_i)}) \geqslant \epsilon$. Set
\[
U = \limsup \overline{H(1,x_i)} = \bigcap_{n} \bigcup_{i \geqslant n} \overline{H(1,x_i)}
\]
i.e. the set $U$ consists of all points in $\fmin$ which lie in infinitely many $\overline{H(1,x_i)}$. The sets $U_n = \cup_{i \geqslant n} \overline{H(1,x_i)}$ form a decreasing sequence i.e. $U_n \supseteq U_{n+1}$. Moreover, $\nu(U_n) \geqslant \epsilon$ for all $n$. So $\nu(U) \geqslant \epsilon$, which implies that $U$ is non-empty.

Let $\lambda \in U$, and pass to a subsequence $x_i$ such that $\lambda \in \overline{H(1,x_i)}$ for all $i$ in the subsequence. We claim that $\cap \overline{H(1,x_i)} = \{ \lambda \}$. Suppose $\xi$ is a minimal foliation that also lies in $\cap \overline{H(1,x_i)}$.

Let $y$ and $z$ be any points in $H(1,x_i)$. Let $p$ and $q$ be the points on $[1,x_i]$ closest to $y$ and $z$ respectively. By the triangle inequality, we have $\widehat{d}(y,z) \leqslant \widehat{d}(y,p) + \widehat{d}(p,q) + \widehat{d}(z,q)$. Hence, the Gromov product $(y \mid z)$ satisfies
\[
(y \mid z) \geqslant \frac{1}{2} \left( \widehat{d}(1,y) - \widehat{d}(y,p) + \widehat{d}(1,z) - \widehat{d}(z,q) + \widehat{d}(p,q) \right)
\]
By Proposition 3.2 from \cite{Mah2}, we have $\widehat{d}(1,y) - \widehat{d}(y,p) \geqslant \widehat{d}(1,p) - 6\delta$ and $\widehat{d}(1,z) - \widehat{d}(z,q) \geqslant \widehat{d}(1,q) - 6\delta$. Hence
\[
(y \mid z) \geqslant \frac{1}{2} \left( \widehat{d}(1,p) + \widehat{d}(1,q) - \widehat{d}(p,q) - 12\delta \right)
\]
Now, either $\widehat{d}(1,q)- \widehat{d}(p,q)= \widehat{d}(1,p)$ or $\widehat{d}(1,p) - \widehat{d}(p,q)= \widehat{d}(1,q)$. Without loss of generality, assuming that the former is true we get
\[
(y \mid z) \geqslant \frac{1}{2} \left( 2 \widehat{d}(1,p) - 12 \delta \right)
\]
By Proposition~\ref{half}, we have $\widehat{d}(1,p) \geqslant (1/2)\widehat{d}(1,x_i) - 3 \delta$. So
\[
(y \mid z) \geqslant \frac{1}{2}\widehat{d}(1,x_i) - 9\delta
\]
Thus the Gromov product tends to infinity as $i$ goes to infinity implying $\xi = \lambda$.

By assumption, the harmonic measure $\nu(\overline{H(1,x_i)}) \geqslant \epsilon$ for all $i$ in the subsequence. This implies that $\nu(\lambda) \geqslant \epsilon$. But then, the measure $\nu$ has atoms, which contradicts Theorem~\ref{Poisson}. Therefore $\nu(\overline{H(1,x_i)}) \to 0$ as $\widehat{d}(1,x_i) \to \infty$.
\end{proof}

\noindent In fact, Maher proves that the harmonic measure $\nu(\overline{H(1, x)})$ decays exponentially in $\widehat{d}(1,x)$. To be precise,

\begin{proposition}[Maher \cite{Mah3} Lemma 5.4] \label{half-decay}
Let $\mu$ be a finitely supported probability distribution on $G$ whose support generates a non-elementary subgroup, and let $\nu$ be the corresponding harmonic measure. There are positive constants $K_3$ and $L<1$ that depend only on the topology of the surface and $\mu$, such that if $\widehat{d}(1,x) \geqslant K_3$ then $\nu(\overline{H(1,x_i)}) \leqslant  L^{\widehat{d}(1,x)}$.
\end{proposition}

\section{Maher's theorem}\label{Maher}

Roughly speaking, Maher's theorem states that for any sub-surface projection, the probability that the random walk nests distance $n$ in the sub-surface projection away from the base point is exponentially small in $n$.

\subsection{Application of the bounded geodesic image theorem:}
The important tool is the following bounded geodesic image theorem of Masur and Minsky (Theorem 3.1 in \cite{Mas-Min2}), which says that a geodesic in $\cC(\Sigma)$ for which sub-surface projection to $Y$ is defined for every vertex projects to a set of bounded diameter in $\cC(Y)$, where the bound depends only on the topology.

\begin{theorem}[Bounded geodesic image]\label{bounded-image}
Let $Y$ be an essential connected subsurface of $\Sigma$, not a three-punctured sphere, and let $\gamma$ be a geodesic segment in $\cC(\Sigma)$, such that $\pi(x) \neq \varnothing$, for every vertex $x \in \gamma$. Then, there is a constant $M_Y$, which depends only on the topological type of $Y$, such that the diameter of $\pi(\gamma)$ is at
most $M_Y$.
\end{theorem}

In particular, since there are only finitely many topological types of subsurfaces $Y$ in $\Sigma$, we can choose $M$ to be the maximum  over all $M_Y$. The constant $M$ is called the Masur-Minsky constant. Since the relative space $\widehat{G}$ is quasi-isometric to $\cC(Y)$, the bounded geodesic image theorem holds in $\widehat{G}$ also. Since, for the remainder of this section, we work in $\widehat{G}$, we continue to denote the Masur-Minsky constant in $\widehat{G}$ by $M$. We will think of a geodesic $\gamma$ in $G$ as a function $\gamma: \Z \to G$, and we will write $\gamma_n$ for $\gamma(n)$.

The set of simple closed curves in $\Sigma$ that are distance at most one from the set of boundary curves $\partial Y$ of $Y$, is a set of diameter at most 3 in $\cC(\Sigma)$. Let $\widehat{N}_Y$ be the pre-image of this set in $\widehat{G}$, under the quasi-isometry between $\widehat{G}$ and $\cC(Y)$. Then there is a positive constant $R_2$ such that the set $\widehat{N}_Y$ has diameter at most $3R_2$ in $\widehat{G}$. Recall from~\ref{Proj} that by defining the projection $\pi$ on $\widehat{G}$ using the marking complex, we ensure that $\pi$ is defined at all points of $\widehat{N}_Y$. However, the bounded geodesic image theorem~\ref{bounded-image} fails to hold for geodesics in $\widehat{G}$ that pass through $\widehat{N}_Y$.

Let $y_0 \in \cC(Y)$ be the image of the identity element $1$ in $G$, under the sub-surface projection $\pi$ i.e. $y_0 = \pi(1)$.

\begin{lemma}\label{measure}
There is a constant $K_4$, which depends only on $\delta$, such that there is a finite collection of half-spaces $H(1, x_i)$ with $\widehat{d}(1,x_i) = K_4$, such that for any subsurface $Y$ of $\Sigma$, the union of the half-spaces disjoint from $\widehat{N}_Y$ has harmonic measure at least $1/2$.
\end{lemma}

\begin{proof}
First, we want $K_4 > K_3$, so that Proposition~\ref{half-decay} can be applied. Second, by choosing $K_4$ sufficiently large such that $L^{K_4 -  3R_2 - K_1 - K_2} < 1/4$, we can bound the harmonic measure of any half-space $H(1,a)$ with $\widehat{d}(1,a) = K_4$, from above by 1/4. The collection of endpoints of geodesic rays based at $1$ is dense in $\fmin$. This implies that for any $K_4 > 0$,
\[
\nu(\cup_{\widehat{d}(1,x) = K_4} \overline{H(1, x)}) = 1
\]
Therefore, for any $\epsilon > 0$, there is a finite collection $\{H(1, x_i) \}_{1 \leqslant  i \leqslant  N}$ of half-spaces, with $\widehat{d}(1,x_i) = K_4$, such that $\nu(\cup_i \overline{H(1, x_i)} \geqslant 1 - \epsilon$. It will be convenient
for us to choose $\epsilon = 1/4$. Suppose $\widehat{N}_Y$ hits some half-space $H(1, x_i)$ in the finite collection. Since $K_4 \geqslant 3R_2 + K_1$, we can apply Proposition~\ref{boundedset} to conclude there is a $y \in [1, x_i]$ with $\widehat{d}(1,y) = K_4 - 3R_2 - K_1$ such that the union $H(1, x_i) \cup \widehat{N}_Y$ belongs to the half-space $H(1, y)$. By Proposition~\ref{nbhd}, there is a half-space $H(1,z)$ with $\widehat{d}(1,z) = K_4 - 3R_2 - K_1 - K_2$ such that any half-space $H(1, x_j)$ with $\widehat{d}(1,x_j) \geqslant K_4$, intersecting $H(1, y)$, is contained in $H(1, z)$. By Proposition~\ref{half-decay}, the harmonic measure of half-spaces decays exponentially in the relative distance. So
$\nu(\overline{H(1, z)} \leqslant  L^{K_4 - 3R_2 - K_1 - K_2}< 1/4$. So, the measure of the half-spaces disjoint from $\widehat{N}_Y$ is at least 1/2.
\end{proof}

\begin{lemma}\label{half-bounded}
Given the finite collection of half-spaces as in Lemma~\ref{measure}, there is a constant $K_5$, depending on the collection, such that for any subsurface $Y$, the projection of the union of the half-spaces disjoint from $\widehat{N}_Y$ is contained in the $K_5$-neighborhood of $y_0= \pi(1)$ in $\cC(Y)$.
\end{lemma}

\begin{proof}
Every point in a half-space $H(1,x_i)$ can be connected to 1 by a piecewise geodesic path with at most two pieces: the geodesic $[1,x_i]$ followed by a geodesic connecting $x_i$ to the point. Consider half-spaces in the collection that are disjoint from $\widehat{N}_Y$. If $\widehat{N}_Y$ does not hit the geodesic $[1, x_i]$, then by Theorem~\ref{bounded-image}, the projection of the union $H(1, x_i) \cup [1,x_i]$ has bounded image in $\cC(Y)$ with diameter at most $2M$. So suppose that $\widehat{N}_Y$ hits the geodesic $[1, x_i]$. Since there is a fixed collection of finitely many such geodesics, we can set
\[
M' = \max \left(\max_i(d(1,x_i)), M, A \right)
\]
where $d$ is the actual distance in $G$ and $A$ is the cutoff in the quasi-distance formula~\ref{quasi-distance}. By the quasi-distance formula, up to a universal constant that depends only on the topology of $\Sigma$, for any sub-surface $Y$, the projection of all geodesic segments $[1,x_i]$ lies in a $M'$-neighborhood of $y_0$. So we may choose $K_5 = M + M'$ to conclude the proof of the lemma.
\end{proof}

\subsubsection{Exponential decay:}

Recall that we have chosen a base-point $y_0$ in $\cC(Y)$ to be the image under $\pi$ of the identity element $1 \in G$. Let $d_Y$ denote the metric in $\cC(Y)$.

Given a subset $A$ in $\cC(Y)$, we will write $A^c$ to be the complement of $A$ in $\cC(Y)$. We say a pair of sets $A_1 \supset A_2$ in $\cC(Y)$ is {\em $K$-nested} if $d_Y(A_1^c, A_2) \geqslant K$. We say a nested collection of sets $A_1 \supset A_2 \supset \cdots$ is $K$-nested, if each adjacent pair $A_i \supset A_{i+1}$ is $K$-nested in $\cC(Y)$.

A pair of sets $B_1 \supset B_2$ in $\widehat{G}$ is $K$-nested for $\pi$ if the image pair $A_1 = \pi(B_1)$ containing $A_2 = \pi(B_2)$ is $K$-nested in $\cC(Y)$. It should be pointed out that such a pair $B_1 \supset B_2$ is not necessarily $K$-nested for the relative metric $\widehat{d}$ on $\widehat{G}$.

The main theorem due to Maher \cite{Mah4} is

\begin{theorem}[Exponential Decay]\label{exp-decay}
Let $\mu$ be a probability distribution on $G$, with finite support, such that the sub-group generated by the support is non-elementary. Let $Y$ be a sub-surface of $\Sigma$. Then there is a constant $K_6$, which depends on $\mu$ but is independent of the sub-surface $Y$, such that if $\pi(1) \notin A_1 \supset A_2 \supset \cdots$ is a collection of $K_6$-nested subsets in $\cC(Y)$, then $\nu(\overline{\pi^{-1}(A_k)}) \leqslant  (1/2)^k$.
\end{theorem}

Let $H(1, x_i)$ be a collection of half-spaces as in Lemmas~\ref{measure} and~\ref{half-bounded} above, and let $K_5$
be the corresponding constant from Lemma~\ref{half-bounded}. Let $\nu_g$ for the harmonic measure for starting at base-point $g$ instead of $1$, so $\nu_g(X) = \nu(g^{-1}X)$.

\begin{lemma}\label{half-estimate}
For any sub-subsurface $Y$,
\[
\nu_g\left(\overline{\pi^{-1}(B_{K_5}(\pi(g)))}\right) > \frac{1}{2}
\]
where $B_{K_5}(\pi(g))$ is the ball in $\cC(Y)$ of radius $K_5$, centered at $\pi(g)$.
\end{lemma}

\begin{proof}
Consider the projection into $\cC(Y)$ of the half-spaces $H(g, g x_i)$ disjoint from $\widehat{N}_Y$. By Lemma~\ref{half-bounded}, the projection of the union of these half-spaces lies in a $K_5$-neighborhood of $\pi(g)$ in $\cC(Y)$. By Lemma~\ref{measure}, $\pi^{-1}(B_{K_5}(\pi(g)))$ contains all the half-spaces $H(g , g x_i)$ disjoint from $\widehat{N}_Y$, and so has measure at least $1/2$.
\end{proof}

Let $D$ be the diameter of the support of $\mu$. We choose $K_6 > K_5 + 2D$. We can now do conditional measure on $B_i = \pi^{-1}(A_i)$ of $K_6$-nested collections of sets for $\pi$. Given such a nested collection $B_i$ for $\pi$, one can define a ``midpoint set'' to be the set $\{ g \in B_i \mid d_Y(\pi(g), A_{i+1} ) = d_Y(\pi(g), A_i^c ) \} $. However, for the subsequent argument, we need a neighborhood of this which covers a $2R_1D$-neighborhood of the projection in $\cC(Y)$, where recall from the previous section that the map $\pi$ on $G$ is coarsely $2R_1$-Lipshitz. So we make the following definition.

\begin{definition}
Let $\{B_i\}$ be a nested collection of sets for $\pi$ with a large enough nesting distance. Define the {\em midpoint sets} to be
\[
E_i = \{ g \in B_i : \vert d_Y(\pi(g), A_{i+1}) - d_Y(\pi(g), A_i^c )\vert \leqslant  2R_1D \}
\]
where recall from 9.3 that $\pi$ is coarsely $2R_1$-Lipshitz on $G$.
\end{definition}

\begin{proposition} \label{midpoint}
Choose $K_6 > 2R_1+M$, and let $(B_i)$ be a $K_6$-nested collection of sets for $\pi$. If a sample path $\p$ converges to $\overline{B_{i+1}}$, then for some $n$, the point $\omega_n$ lies in a midpoint set $E_i$.
\end{proposition}

\begin{proof}
Suppose $\omega_n$ converges to $\lambda \in \overline{B_{i+1}}$. Then, there is a sequence $l_n \in B_{i+1}$ such that $l_n$ converges to $\lambda$, and so the Gromov product $(\omega_n \vert l_n)$ converges to $\infty$. The foliation $\lambda$ is minimal, so for $n$ large enough, a geodesic joining $\omega_n$ to $l_n$ misses $\widehat{N}_Y$. Theorem~\ref{bounded-image} then implies that $d_Y(\pi(\omega_n), \pi(l_n)) \leqslant  M$. In particular, for all $n$ sufficiently large, $d_Y( \pi(\omega_n), B_{i+1}) \leqslant  M$ i.e. the projection of $\omega_n$ has to get within distance $M$ of $A_{i+1}$.

As $\mu$ has finite support, the sample path satisfies $d(\omega_n, \omega_{n+1}) \leqslant  D$, where $d$ is the actual distance in $G$. Since $\pi$ is coarsely $2R_1$-Lipshitz, we have $d_Y(\pi(\omega_n, \omega_{n+1})) \leqslant  2R_1D$. Therefore, if the nesting distance satisfies $K_6 > 2R_1 + M$, then there is some $\omega_n$ such that
\[
\vert d_Y(\pi(\omega_n) , A_{i+1}) - d_Y(\pi(\omega_n), A_i^c )\vert \leqslant  2R_1D
\]
i.e. there is an $\omega_n \in E_i$.
\end{proof}

\begin{proof} (of Theorem~\ref{exp-decay})
We will compute the probability a sample path converges into $\overline{B_{k+1}}$, given that it converges into $\overline{B_k}$. By Proposition~\ref{midpoint}, any sample path $\p$ that converges into $\overline{B_{k+1}}$ has to hit the midpoint set $E_k$. So we can condition on an element $g$ in $E_k$.

Consider the collection of all random walks starting at $g$. Since $\fmin$ is the disjoint union of the sets $\overline{B_k^c} \cup (\overline{B_k} \setminus \overline{B_{k+1}}) \cup \overline{B_{k+1}}$, a random walk converges into precisely one of these sets. Set $p_1 = \nu_g(\overline{B_k^c}), p_2 = \nu_g(\overline{B_k} \setminus \overline{B_{k+1}}), p_3 = \nu_g(\overline{B_{k+1})}$. Then $p_1 + p_2 +
p_3 = 1$. We want to give an upper bound for the relative probability
\[
\frac{ \nu_g(\overline{B_{k+1}}) }{ \nu_g(\overline{B_k}) } = \frac{ p_3 }{ p_2 +p_3  }
\]
i.e. the probability that a sample path starting from $g$ converges into $\overline{B_{k+1}}$, given that it converges into $\overline{B_k}$. The nesting distance $K_6$ can be chosen to be larger than $K_5+D$. Since the sets $B_i$ are $K_6$-nested for $\pi$, the ball $B_{K_5}(\pi(g))$ in $\cC(Y)$ of radius $K_5$ centered at $\pi(g)$, is contained in $A_k \setminus A_{k+1}$. By Lemma~\ref{half-estimate}, the measure $p_2 = \nu_g(\overline{B_k} \setminus \overline{B_{k+1}}) > 1/2$. This implies that for all $k$
\begin{equation} \label{key}
\frac{ \nu_g(\overline{B_{k+1}}) }{ \nu_g(\overline{B_k}) } = \frac{p_3}{p_2 + p_3} < \frac{p_3}{1/2 + p_3}
= 1 - \frac{1}{2}\left( \frac{1}{1/2 + p_3} \right) = \frac{1}{2}
\end{equation}
Using Estimate~\eqref{key} inductively we have $\nu(\overline{B_k}) = \nu(\overline{\pi^{-1}(A_k)}) < (1/2)^k$ as required.
\end{proof}

\section{Sub-surface projections and the push-in sequence}\label{Pushin}

We switch back to the curve complex $\cC(\Sigma)$. For any train track $\tau$, let $T(\tau)$ be the set of simple closed curves carried by $\tau$.

At this point, the key ideas for the construction of the singular set are in place: Start by embedding in $\Sigma$ a non-classical interval exchange with combinatorics $\phi_0$ constructed in Section~\ref{Combinatorics}. Call this the initial interval exchange and denote it by $J$. Choose the embedding such that the base-point $\bb$ is outside $T(J)$. By Theorem~\ref{Uniform-distortion}, almost every expansion from $J$ becomes $C$-distributed infinitely often with each instance of $C$-distribution having combinatorics $\phi_0$. In particular, the union of stages that are the $m$-th instances of $C$-distribution has full Lebesgue measure. Follow each such $C$-distributed stage by $n$ successive Dehn twist sequences, to get a collection of stages. The union of stages in this collection should be the set $Y^{(m)}_n$, which figures in Theorem~\ref{BC2}. The estimate~\eqref{vol} implies that for the Lebesgue measure, each stage in the collection has proportion $\approx 1/n^j$ in the corresponding $C$-distributed stage. Taking union over the $C$-distributed stages implies $\ell(Y^{(m)}_n) \approx 1/n^j$, as required in Theorem~\ref{BC2}. Finally, for each $C$-distributed stage in question, consider the sub-surface projection to it's stable vertex cycle. In defining $Y^{(m)}_n$, we followed the $C$-distributed stage by the Dehn twist sequence $n$ times. This should increase sub-surface projections to the stable vertex curve by $n$. So for the harmonic measure, one expects to show by Theorem~\ref{exp-decay}, that the stage obtained after the twists has proportion $\leqslant \exp(-n)$ in the $C$-distributed stage. Taking union over all the $C$-distributed stages would then imply $\nu(Y^{(m)}_n) \leqslant \exp(-n)$.

Except the problem is that the above setup is not the correct one for Theorem~\ref{exp-decay}. To get the exponential decay estimate, Theorem~\ref{exp-decay} requires that a suitable pre-image under the sub-surface projection to the stable vertex cycle, sits inside the $C$-distributed stage in question. In the above setup, this is not true. So we have to finesse a little bit, to get a situation to which Theorem~\ref{exp-decay} applies. In this section, we give the technical details necessary for that.

\subsection{Sub-surface projections to carried curves:} A train track is said to be {\em generic} when every switch in it is trivalent. Let $\tau$ be a generic complete train track. The following proposition proves that the sub-surface projection of the complement $\cC(\Sigma) \setminus T(\tau)$ to an annulus with core curve ``deeply carried'' by $\tau$ is has a universally bounded diameter.

\begin{proposition}\label{proj-deeply}
Let $x \in T(\tau) $ be carried by $\tau$ such that it passes over every branch of $\tau$ at least thrice. Let $\pi$ be the sub-surface projection to the annulus $A$ with core curve $x$. Then 
\[
{\text diam}\left( \pi (\cC(\Sigma) \setminus T(\tau)) \right) \leq 5
\]
\end{proposition}

\begin{proof}
The main ingredient of the proof is the apparatus of {\em efficient position} developed by Masur, Mosher and Schleimer \cite{Mas-Mos-Sch}.

\subsubsection{Efficient position:} Given a train track $\tau$, let $N(\tau)$ be a neighborhood of a train track foliated by ties. For non-classical interval exchanges, this is just our picture of the interval with bands. For a generic track $\tau$, a curve $c$ is said to be in efficient position with respect to $\tau$, if
\begin{enumerate}
\item every component of $c \cap N(\tau)$ is either a tie or carried by $\tau$,
\item every region in $\Sigma \setminus (c \cup N)$ has negative index or is a rectangle. (See \cite{Mas-Mos-Sch} for the precise definition of index)
\end{enumerate} 
For our purposes, all we need is the implication that if $c$ is in efficient position with respect to $\tau$, then there is no embedded bigon in the complement $\Sigma \setminus (c \cup \tau)$. 
The main theorem of Masur, Mosher and Schleimer is 

\begin{theorem}(Theorem 4.1 in \cite{Mas-Mos-Sch}:) 
Let $\tau$ be a birecurrent generic train track and suppose $c$ is a non-peripheral simple closed curve. Then efficient position of $c$ with respect to $\tau$ exists and is unique up to rectangle swaps, annulus swaps, and isotopies of $\Sigma$ preserving the foliation of $N(\tau)$ by ties.
\end{theorem}

Going back to the proof of Proposition~\ref{proj-deeply}, let $y$ and $z$ be curves in $\cC(\Sigma) \setminus T(\tau)$. By the theorem of Masur-Mosher-Schleimer, the curves $y$ and $z$ can be put in efficient position with respect to $\tau$. Since $\tau$ is complete, every efficient position of $y$ meets some branch of $\tau$ dually i.e. the intersection with the branch is a tie. Similarly for every efficient position of $z$. Let $\widehat{\Sigma}$ be the compactified annulus cover of $\Sigma$ corresponding to $A$. There are countably many lifts of $x$ to $\widehat{\Sigma}$; exactly one of these is the core curve of $\widehat{\Sigma}$. Call this lift $X_0$. All other lifts of $x$ are inessential arcs with both endpoints on the same boundary component of $\widehat{\Sigma}$.  

\begin{figure}[htb]
\begin{center}
\ \psfig{file=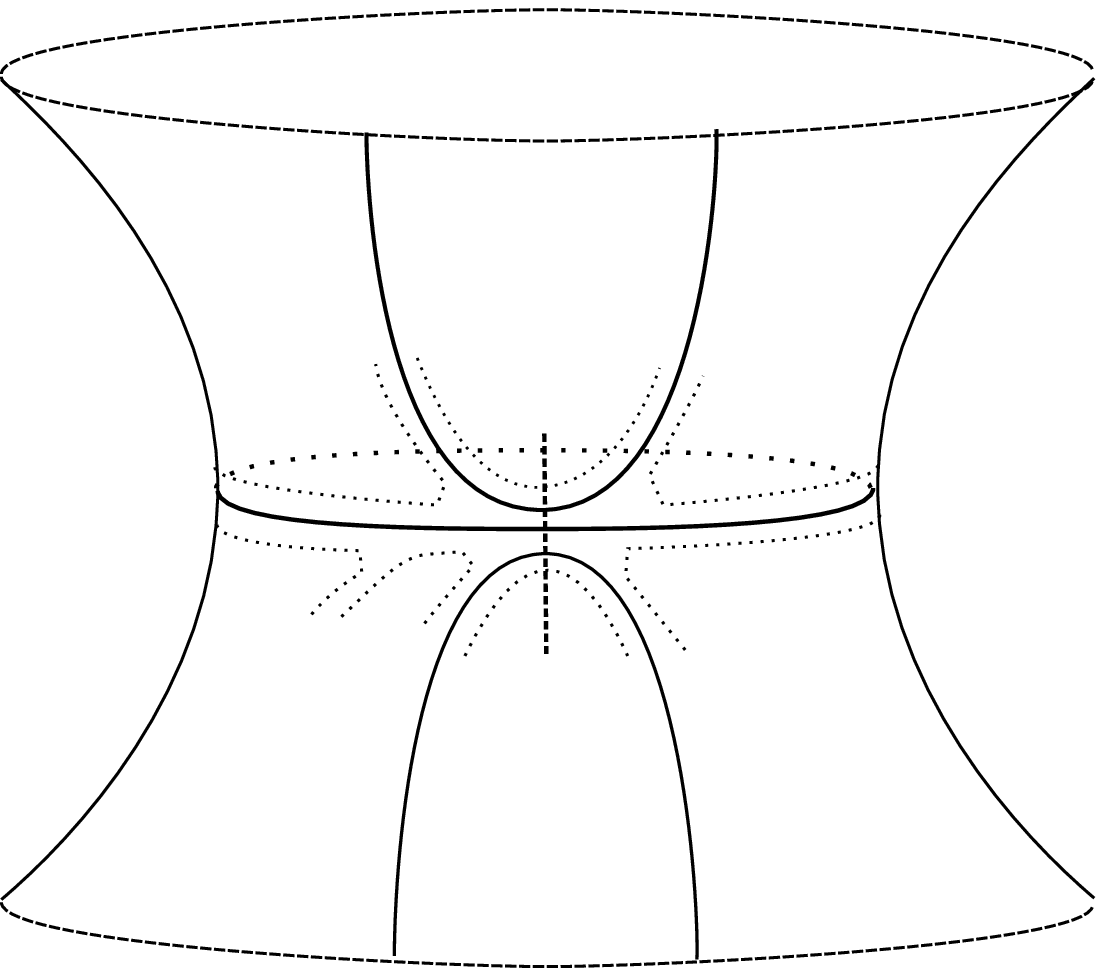, height=2.5truein, width=2.5truein} \caption{Lifts to $\widehat{\Sigma}$} \label{lift}
\end{center}
\setlength{\unitlength}{1in}
\begin{picture}(0,0)(0,0)
\put(0,2.1){$S$}\put(-0.6,2.6){$X_+$}\put(-0.6,0.8){$X_-$}\put(-0.6,1.4){$b$}\put(-1,1.8){$X_0$}\put(1.3,3){+}\put(1.3,0.6){-}
\end{picture}
\end{figure}

Let $s$ be the tie that is the intersection of an efficient position of $y$ with a branch of $\tau$. Choose a lift $S$ of the tie $s$ meeting $X_0$ such that the point of intersection $S \cap X_0$ is sandwiched in between intersections of $S$ with other lifts of $x$ on either side of $X_0$. See Figure~\ref{lift}.

On either side of $X_0$, let $X_+$ and $X_-$ be the ``outermost'' lifts of $x$ that $S$ intersects i.e.  the lifts $X_+$ and $X_-$ are the farthest to the point $S  \cap X_0$ along $S$. For instance, the dotted lines in Figure~\ref{lift} show a portion about $s$ of the tie neighborhood $N(\tau)$ lifted up to $\widehat{S}$. The condition that $x$ passes over each branch at least thrice implies that there is at least one lift of $x$ that passes through the branch marked $b$ and intersects $S$. But as is clear from the picture, a lift through $b$ is not the outermost one in the above sense. 

Since the lifts $X_{\pm}$ are inessential arcs, they cut off discs $D_{\pm}$ in $\widehat{\Sigma}$. In Figure~\ref{lift}, the discs $D_\pm$ can be seen to be bounded by $X_\pm$ and the boundary components marked $\pm$. The lift $Y$ of $y$ extending $S$, cannot intersect $X_{\pm}$ again; otherwise, there is a bigon in the complement $\Sigma \setminus (y \cup \tau)$ violating efficient position. This implies that the lift $Y$ is contained in $D_+ \cup S \cup D_-$. 

We can similarly consider an efficient position of $z$, and assume that it intersects $\tau$ dually in a tie $t$ different from $s$. Lifting $t$ up to $\widehat{\Sigma}$ to get a suitable arc $T$, we  repeat the argument above to show that there exists lifts $X'_{\pm}$ of $x$ cutting off discs $D'_{\pm}$ in $\widehat{\Sigma}$, such that a lift $Z$ of $z$ extending $T$ is contained in $D'_+ \cup T \cup D'_-$.  

Fixing the endpoints of $Y$ and $Z$, it follows from the above containments that the algebraic intersection number $Y \cdot Z$ is at most 2. Recall from the initial part of Section~\ref{Projections}, that the distance between $Y$ and $Z$ in the complex $\cC(x)$ is given by $1+ \vert Y \cdot Z \vert$, thus finishing the proof of Proposition~\ref{proj-deeply}. 
\end{proof}

Even though Proposition~\ref{proj-deeply} is stated for generic complete train tracks, the result is also true for non-classical interval exchanges. We can {\em comb} (See Section 1.4 of \cite{Pen-Har}) a non-classical interval exchange moving left to right along the base interval to get a generic train track. For example, see the first picture in Figure 21 of \cite{Dun-DTh}. It can be directly checked that the resulting track is transversely recurrent. Then, by Proposition 1.4.1 of \cite{Pen-Har}, the resulting generic train track is complete, and so Proposition~\ref{proj-deeply} applies to it. The operation of combing is isotopic to identity, and hence the set of carried curves remains unchanged.

\subsection{The push-in sequence:} \label{discuss-pushin} Now we focus on non-classical interval exchanges $\jmath$ that have the combinatorial type $\phi_0$ constructed in Section 7. A directed path $\kappa$ in $\overline{\cG}$ starting at $\phi_0$ and terminating in $\phi_0$ can be realized as a splitting sequence of such $\jmath$. We denote the resulting non-classical interval exchange by $\jmath \ast \kappa$. Let $T(\jmath \ast \kappa)$ be the set of simple closed curves carried by $\jmath \ast \kappa$.

By the nesting lemma of Masur and Minsky viz. Lemma 4.7 of \cite{Mas-Min1}, it is possible to combinatorially fix a directed path $\kappa$ in $\overline{\cG}$, starting from and terminating at $\phi_0$, such that the stable vertex cycle $v(\jmath \ast \kappa)$ passes over every band in $\jmath$ at least thrice. 

Now consider the subsurface projection $\pi$ to the annulus with core curve $v(\jmath \ast \kappa)$. Fix a quasi-isometry of $\cC(v(\jmath \ast \kappa))$ with $\Z$, so that we can assume $\pi$ takes values in $\Z$. The quasi-isometry constant depends only on the topology of $\Sigma$. By composing with a translation of $\Z$, we can assume that $\pi(\bb) =0$. By Proposition~\ref{proj-deeply}, the projection $\pi$ of every point in $\cC(\Sigma) \setminus T(\jmath)$ is within distance $5$ of $\pi(\bb)$. This implies that the pre-image $\pi^{-1}\left((-\infty, -5) \cup (5, \infty)\right)$ is contained in $T(\jmath)$. It is in this sense that $\kappa$ is {\em the push-in sequence}.

\section{The singular set}\label{Singular}

\subsection{Choosing the initial chart:}

An embedding of an interval exchange with combinatorics $\phi_0$ into $\Sigma$, as a complete train track with a single vertex, identifies the initial configuration space $W_0$ with a chart in $\pmf$. We choose the initial embedding such that the base-point $\bb$ does not belong to the set $T(J)$ of simple closed curves carried by the embedded interval exchange $J$.

The construction of the singular set will proceed inside the chart of $\pmf$ given by $J$.

\subsection{Relative probability that the sequence $\kappa \ast n\jmath_0$ follows a $C$-distributed stage:} By Lemma~\ref{nciem-control}, the relative probability that the push-in sequence $\kappa$ follows a $C$-distributed stage $\jmath$ is up to a universal constant $c>1$, the same as the probability $\ell(JQ_\kappa(W_0))$ that an expansion begins with $\kappa$ i.e.
\[
\frac{1}{c} \ell(JQ_\kappa(W_0)) < \frac{\ell(JQ_{\jmath \ast \kappa}(W_0))}{\ell(JQ_\jmath(W_0))} < c \cdot \ell(JQ_\kappa(W_0))
\]
Since $\kappa$ is a priori fixed, whenever the sequence $\kappa$ follows a $C$-distributed stage $\jmath$, the resulting stage $\jmath \ast \kappa$ is $C'$-distributed, for some $C'$ that depends only on $C$ and $d$. Denote the matrix associated to the sequence $\jmath \ast \kappa \ast n \jmath_0$ by $Q_{\jmath,n}$. Again by Lemma~\ref{nciem-control}, there exists a constant $a_1$ that depends only on $a_0$ and $C'$, such that
\[
\frac{1}{a_1n^j} < \frac{\ell (JQ_{\jmath,n}(W_0))}{\ell(JQ_{\jmath \ast \kappa}(W_0))} < \frac{a_1}{n^j}
\]
Hence the relative probability that the sequence $\kappa \ast n \jmath_0$ follows a $C$-distributed stage $\jmath$ satisfies
\begin{equation}\label{relative-prob}
\frac{1}{ca_1n^j} < \frac{\ell (JQ_{\jmath,n}(W_0))}{\ell(JQ_\jmath(W_0))}< \frac{ca_1}{n^j}
\end{equation}
i.e. the relative probability $\ell (JQ_{\jmath,n}(W_0))/\ell(JQ_\jmath(W_0)) \approx 1/n^j$.

\subsection{Construction of the doubly indexed sequence of sets $Y^{(m)}_n$:} By Theorem~\ref{Uniform-distortion}, almost every expansion becomes $C$-distributed infinitely often. Hence, for every non-negative integer $m$, almost every expansion has a stage that is the $m$-th instance of $C$-distribution. When $m=0$, we mean the initial stage itself with no splitting whatsoever.

Let $S_m$ be the set of stages $\jmath$ that are the $m$-th instances of $C$-distribution. For $\jmath \in S_m$, let $Y_\jmath$ be the set of $\x$ in $W_0$ whose expansion begins with $\jmath$ i.e. the subset $Y_\jmath = JQ_\jmath(W_0)$. For distinct $\jmath, \bar{\jmath}$ in $S_m$, the sets $Y_\jmath, Y_{\bar{\jmath}}$ have disjoint interiors. By Theorem~\ref{Uniform-distortion}, the union over all $\jmath \in S_m$ of the sets $Y_\jmath$ is a set of full measure i.e.
\begin{equation}\label{full-measure}
\sum_{\jmath \hskip 2pt \in S_m} \ell(Y_\jmath) = 1
\end{equation}
Follow each $\jmath$ in $S_m$ by the sequence $\kappa \ast n \jmath_0$, and let $Q_{\jmath,n}$ be the matrix associated to $\jmath \ast \kappa \ast n \jmath_0$. Let $Y_{\jmath,n} = JQ_{\jmath,n}(W_0)$. Then, the ratio $\ell(Y_{\jmath,n})/ \ell(Y_\jmath)$ satisfies Estimate~\eqref{relative-prob}. Let $Y^{(m)}_n$ be the union
\[
Y^{(m)}_n = \bigcup_{\jmath \hskip 2pt \in S_m} Y_{\jmath,n}
\]
For $m=0$, the set $Y^{(0)}_n$ is just $JQ_{n \jmath_0}(W_0)$ i.e the set of $\x$ whose expansion begins with $n \jmath_0$. First, we estimate the Lebesgue measure of $Y^{(m)}_n$.

\begin{lemma}
\begin{equation} \label{poly-decay-estimate}
\frac{1}{ca_1n^j} < \ell(Y^{(m)}_n) < \frac{ca_1}{n^j}
\end{equation}
\end{lemma}

\begin{proof}
Write
\[
\ell(Y^{(m)}_n) = \sum_{\jmath \hskip 2pt \in S_m} \ell(Y_{\jmath,n})
= \sum_{\jmath \hskip 2pt \in S_m} \ell(Y_\jmath) \frac{\ell(Y_{\jmath,n})}{\ell(Y_\jmath)}
\]
Estimate~\eqref{relative-prob} for the ratio $\ell(Y_{\jmath,n})/ \ell(Y_\jmath)$, and Equation~\eqref{full-measure} finishes the proof.
\end{proof}
\noindent In particular, the above lemma shows that for any $m_1, m_2$,
\begin{equation}\label{equal}
\ell(Y^{(m_1)}_n) \approx \ell(Y^{(m_2)}_n)
\end{equation}

\subsection{Almost independence of $Y^{(m)}_n$ for the Lebesgue measure:} Let $(m_1,n_1)$ and $(m_2,n_2)$ be a pair of indices with $m_1 < m_2$. Since the sets $Y_\iota: \iota \in S_{m_1}$ is a partition of a set of full measure, it is enough to check that almost independence holds in each $Y_\iota$. For any $\iota \in S_{m_1}$ and $\jmath \in S_{m_2}$, either $Y_\jmath$ is contained in $Y_{\iota,n_1}$ or has interior disjoint from it. By Lemma~\ref{nciem-control}, given $\iota \in S_{m_1}$, the relative probabilities satisfy
\begin{eqnarray*}
P(Y^{(m_1)}_{n_1}\vert \iota) &\approx& \ell(Y^{(0)}_{n_1})\\
P(Y^{(m_2)}_{n_2}\vert \iota) &\approx& \ell(Y^{(m_2- m_1)}_{n_2})\\
P(Y^{(m_1)}_{n_1}\cap Y^{(m_2)}_{n_2} \vert \iota) &\approx& \ell(Y^{(0)}_{n_1} \cap Y^{(m_2- m_1)}_{n_2})
\end{eqnarray*}
So it is enough to check that for any $m>0$, the sets $Y^{(0)}_{n_1}$ and $Y^{(m)}_{n_2}$ are almost independent. Again, the main point is that for all $\jmath$ in $S_m$, the set $Y_\jmath$ is either contained in $Y^{(0)}_{n_1}$ or has interior disjoint from $Y^{(0)}_{n_1}$. Let $T_m$ be the subset of $S_m$ consisting of those $\jmath$ for which $Y_\jmath$ is contained in $Y^{(0)}_{n_1}$. The union over all $\jmath \in T_m$ of the sets $Y_\jmath$ is a set of full measure in $Y^{(0)}_{n_1}$. Hence, using Estimate~\eqref{relative-prob} and Equation~\eqref{equal}, we get
\begin{eqnarray*}
\ell \left( Y^{(0)}_{n_1} \cap Y^{(m)}_{n_2} \right) &\approx& \sum_{\jmath \hskip 2pt \in T_m} \ell(Y_{\jmath,n_2})\\
&\approx& \sum_{\jmath \hskip 2pt \in T_m} \ell(Y^{(0)}_{n_2})\ell(Y_\jmath)\\
&\approx& \ell(Y^{(m)}_{n_2}) \sum_{\jmath \hskip 2pt \in J_m} \ell(Y_\jmath)\\
&\approx& \ell(Y^{(m)}_{n_2}) \ell(Y^{(0)}_{n_1})
\end{eqnarray*}
showing almost independence.

\subsection{Harmonic measure estimate for $Y^{(m)}_n$:} Recall from Remark~\ref{wind} that the combinatorics $\phi_0$ has the property that simple closed curves carried by an interval exchange with combinatorics $\phi_0$ can wind around the stable vertex cycle $v$ in one direction only. Let $\pi$ be the projection to the annulus with core curve $v$ and use the quasi-isometry between $\cC(v)$ and $\Z$ to think of $\pi$ as a map to $\Z$. Since simple closed curves carried by the interval exchange can wind around $v$ in one direction only, the projection of the set of carried curves is a one-sided interval of the form $[M_1, \infty)$ or $(-\infty, M_1]$.

Assume that for the initial interval exchange $J$, the base-point $\bb$ intersects the stable vertex cycle $v(J \ast \kappa)$, such that $\pi_{J \ast \kappa}(T(J \ast \kappa))$ is the one sided interval $[M_1, \infty)$, where $\pi_{J \ast \kappa}$ is the sub-surface projection to the annulus with core curve $v(J \ast \kappa)$. Here it is assumed that by composing with a translation of $\Z$ we have arranged that $\pi(\bb) =0$.

For $\jmath \in S_m$, let $v(\jmath \ast \kappa)$ be the stable vertex cycle of the interval exchange given by $\jmath \ast \kappa$. Denote by $\pi_{\jmath \ast \kappa}$, the sub-surface projection to the annulus with core curve $v(\jmath \ast \kappa)$. Fix a quasi-isometry between $\cC(v(\jmath \ast \kappa))$ and $\Z$, so that we can assume that $\pi_{\jmath \ast \kappa}$ takes values in $\Z$ with the projection of the base-point $\pi_{\jmath \ast \kappa}(\bb) = 0$.

The push-in sequence $\kappa$ ensures that the pre-image $\pi_{\jmath \ast \kappa}^{-1}((-\infty, -5) \cup (5, \infty))$ sits entirely inside the set $T(\jmath)$. In addition, the observation in the first paragraph of this subsection says that $\pi_{\jmath \ast \kappa}(T(\jmath \ast \kappa))$ is a one sided interval.

We have the lemma:

\begin{lemma}\label{one-sided}
For all positive integers $m$ and for all $\jmath \in S_m$
\[
\pi_{\jmath \ast \kappa}(T(\jmath \ast \kappa)) \subset [M_1-5,\infty)
\]
where $\pi_{\jmath \ast \kappa}$ is the sub-surface projection to the annulus with core curve $v(\jmath \ast \kappa)$, and $\pi(\bb) = 0$.
\end{lemma}

\begin{proof}
Since $\jmath$ and $J$ have the same combinatorics $\phi_0$, there is a mapping class $g$ that sends the chart given by $J$ to the chart given by $\jmath$ i.e. we have $g^{-1}(T(\jmath)) = T(J)$. It follows that $g^{-1}(T(\jmath \ast \kappa)) = T(J \ast \kappa)$. For the sub-surface projection $\pi_{\jmath \ast \kappa}$ to $v(\jmath \ast \kappa)$, set the origin in $\Z$ by $\pi_{\jmath \ast \kappa}(g \bb)$ instead of $\pi_{\jmath \ast \kappa}(\bb)$. With this choice of origin in $\Z$, we have $\pi_{\jmath \ast \kappa}(T(\jmath \ast \kappa)) = [M_1, \infty)$ because of equivariance. Finally, notice that since $\bb$ lies outside $T(J)$, the point $g \bb$ lies outside $g(T(J))= T(\jmath)$. By Proposition~\ref{proj-deeply}, the projections $\pi_{\jmath \ast \kappa}(g \bb)$ and $\pi_{\jmath \ast \kappa}(\bb)$ are within distance 5 of each other finishing the argument.
\end{proof}

For any interval exchange $\jmath \in S_m$, to define the set $Y_{\jmath, n}$, we followed the sequence $\jmath \ast \kappa$ by $n \jmath_0$. This means that the set $T(\jmath \ast \kappa \ast n \jmath_0)$ is obtained from the set $T(\jmath \ast \kappa)$ by $n$ positive Dehn twists in the stable vertex cycle $v(\jmath \ast \kappa)$. By the equivariance property~\eqref{equivariance}, this increases the projection under $\pi_{\jmath \ast \kappa}$ by $n$. Let $a(n)$ denote the greatest integer less than or equal to $(n+M_1 - 10)/K_6$, where $K_6$ is the nesting distance required in Theorem~\ref{exp-decay}. For $n$ large enough, Theorem~\ref{exp-decay}, Lemma~\ref{one-sided} and the push-in property of $\kappa$, in particular, the fact that $Y_\jmath \supset \pi^{-1}_{\jmath \ast \kappa}(5, \infty)$ imply that the harmonic measures satisfy the estimate
\begin{equation*}
\frac{\nu(Y_{\jmath, n})}{\nu(Y_\jmath)} \leqslant \frac{\nu(\overline{\pi_{\jmath \ast \kappa}^{-1}[M_1-5 +n,\infty)})}{\nu(\overline{\pi_{\jmath \ast \kappa}^{-1}[5,\infty)})} < \left(\frac{1}{2}\right)^{a(n)}
\end{equation*}
Taking union over all $\jmath \in S_m$ we get
\begin{equation} \label{exp-decay-estimate}
\nu(Y^{(m)}_n) = \sum_{\jmath \hskip 2pt \in S_m} \nu(Y_{\jmath,n}) < \left(\frac{1}{2}\right)^{a(n)} \cdot \sum_{\jmath \hskip 2pt \in S_m} \nu(Y_\jmath) = \left(\frac{1}{2}\right)^{a(n)}
\end{equation}
The number $a(n)$ increases linearly in $n$. Thus we get the exponential decay we want.

\subsection{The singular set:}
We have shown that the doubly indexed sequence of sets $Y^{(n)}_m$ satisfy almost independence and the polynomial decay estimate~\eqref{poly-decay-estimate} for the Lebesgue measure, and the exponential decay estimate~\eqref{exp-decay-estimate} for the harmonic measure. Hence, Proposition~\ref{BC2} constructs a set $X$ that has positive Lebesgue measure and zero harmonic measure. The set
\[
Z  = \bigcup_{g \in G} gX
\]
is a measurable $G$-invariant subset of $\pmf$. By the ergodicity of the action of the mapping class group $G$ on $\pmf$ \cite{Mas} \cite{Ker}, the set $Z$ has full Lebesgue measure. On the other hand, since $Z$ is a countable union of sets with zero harmonic measure, it has zero harmonic measure. This proves Theorem~\ref{singular}.

\subsection{Concluding Remarks:} Theorem~\ref{singular} is true for all initial distributions for which the estimate of Theorem~\ref{exp-decay} holds. This should be a larger set of initial distributions than just the finitely supported ones, but the precise description of such distributions is not clear to us.

\end{document}